%% file: euler_av_dis.tex
\documentclass[12pt]{article}

\usepackage[T1]{fontenc}
\usepackage[english]{babel}
\usepackage{amssymb}
\usepackage{amsmath}
\usepackage{epsfig}
\usepackage{color}
\usepackage{hyperref}
\usepackage[a4paper]{geometry}
\allowdisplaybreaks

\newcommand{\deltasw}{\Delta_{sw}\, }
\newcommand{\nablasw}{\nabla_{\!\! sw}\, }
\newcommand{\nablaswa}{\nabla_{\!\! sw}^\alpha\, }
\newcommand{\divsw}{\div_{\! sw}\, }
\newcommand{\divswd}{{\div}_{\! sw,i+1/2}\, }

\newcommand{\divswa}{\div_{\! sw}^\alpha\, }
\newcommand{\nablaswd}{\nabla_{\!\! sw,i}\, }
\newcommand{\nablaswdepsi}{\nabla_{\!\! sw,i}^{\varepsilon}\, }

\newcommand\R{\mathbb{R}}
\newcommand\1{{\bf 1}}

\renewcommand\div{{\rm div}}

\newcommand\sech{\,\mbox{sech}}

\newtheorem{theorem}{Theorem}[section]

\newtheorem{proposition}[theorem]{Proposition}
\newtheorem{corollary}[theorem]{Corollary}

\newenvironment{proof}[1][Proof]{\begin{trivlist}
\item[\hskip \labelsep {\bfseries #1}] \it }{$\blacksquare$\end{trivlist}}

\newtheorem{remark}[theorem]{Remark}
\graphicspath{{Figures/}}

\def\dsp{\displaystyle}
\def\kxi{\begin{pmatrix}1\\ \xi\end{pmatrix}}
\definecolor{Red}{rgb}{1,0,0}
\definecolor{Blue}{rgb}{0,0,1}

\begin{document}

\title{A robust and stable numerical scheme for a depth-averaged Euler system.}

\author{N.~A\"issiouene, M.-O.~Bristeau, E.~Godlewski and J.~Sainte-Marie}
\date{\today}

\thispagestyle{empty}

\maketitle
\begin{abstract}
We propose an efficient numerical scheme for the resolution of a
 non-hydrostatic Saint-Venant type model. The model is a shallow
 water type approximation of the incompressbile Euler system with free
 surface and slightly differs from the Green-Naghdi model.

The numerical approximation relies on a kinetic interpretation of the
model and a projection-correction type scheme. The hyperbolic part of
the system is approximated using a kinetic based finite volume solver
and the correction step implies to solve an elliptic problem involving
the non-hydrostatic part of the pressure.

We prove the numerical scheme satisfies properties such as positivity,
well-balancing and a fully discrete entropy inequality. The numerical
scheme is confronted with various time-dependent analytical solutions. Notice that the numerical procedure
remains stable when the water depth tends to zero.

\end{abstract}

{\it Keywords~:} shallow water flows, dispersive terms, finite volumes,
finite differences, projection scheme, discrete entropy


\tableofcontents

\section{Introduction}

Non-linear shallow water equations model the dynamics of a shallow,
rotating layer of homogeneous incompressible fluid and are typically
used to describe vertically averaged flows in two or three dimensional domains
in terms of horizontal velocity and depth variations. The classical Saint-Venant system \cite{saint-venant}
with viscosity and friction \cite{saleri,gerbeau,marche} is particularly well-suited for the
study and numerical simulations of a large class of geophysical phenomena
such as rivers, lava flows, ice sheets, coastal domains, oceans or even run-off or avalanches
when being modified with adapted source terms \cite{Bouchut2003531,bouchut,mangeney07}. But the Saint-Venant system is built on the hydrostatic
assumption  consisting in neglecting the
vertical acceleration of the fluid. This assumption is valid for a large class of geophysical flows but is
restrictive in various situations where the dispersive effects~-- such as
those occuring in wave propagation~-- cannot
be neglected. As an example, neglecting the vertical acceleration in granular
flows or landslides leads to significantly overestimate the initial flow velocity
\cite{mangeney05,mangeney_nature1}, with strong implication for hazard
assessment.

The derivation of shallow water type models including the
non-hydrostatic effects has received an extensive
coverage~\cite{green,camassa1,bbm,nwogu,peregrine,JSM_DCDS,JSM_nhyd}
and numerical techniques for the approximation of these models have been
recently proposed~\cite{lannes_marche,duran_marche,JSM_CF,lemetayer}.

In~\cite{JSM_nhyd}, some of the authors have presented an original derivation process of a
  non-hydrostatic shallow water-type model approximating
  the incompressible Euler and Navier-Stokes systems with free
  surface where the closure relations are obtained by a
minimal energy constraint instead of an asymptotic expansion. The model slightly differs from the
well-known Green-Naghdi model~\cite{green}. The purpose of this paper
is to propose a robust and efficient numerical scheme for the model
described in~\cite{JSM_nhyd}. The numerical procedure, based on a
projection-correction strategy~\cite{chorin}, is endowed with
properties such as consistency, positivity, well-balancing and
satisfies a fully discrete entropy inequality. We emphasize that the
scheme behaves well when the water depth tends to zero and hence is
able to treat wet/dry interfaces. As far as the authors know, few numerical
methods endowed with such stability properties have been proposed for
such dispersive models extending the shallow water equations.

The paper is organized as follows. In Section~\ref{sec:model_nh}, we
recall the non-hydrostatic model proposed in~\cite{JSM_nhyd} and we
give a rewritting of the system. A kinetic description of the model is
given in Section~\ref{sec:kin} and it is used to derive the numerical
procedure and to prove its properties that are detailed in Sections~\ref{sec:num}
and~\ref{sec:entropy}. In Section~\ref{sec:num}, we first study the
semi-discrete schemes (in space and in time) and then we establish
some properties of the fully discrete scheme. In
Section~\ref{sec:entropy}, we prove the entropy inequality for the
fully discrete scheme. Stationary/transient analytical solutions
of the model are proposed in Section~\ref{sec:anal} and finally the numerical
scheme is confronted with analytical and experimental measurements.

\section{A depth-averaged Euler system}
\label{sec:model_nh}

Several strategies are possible for the derivation of shallow water type
models extending the Saint-Venant system. A usual process is to
assume potential flows and an extensive literature exists concerning
these
models~\cite{chazel1,chazel2,bonneton_lannes,lannes,lannes1,dutykh}. An
asymptotic expansion, going one step further than the classical
Saint-Venant system is also possible \cite{gerbeau,JSM_DCDS,JSM_M3AS}
but such an approach does not always lead to properly defined and/or
unique closure relations. In this paper, we start from a 
non-hydrostatic model derived and studied in~\cite{JSM_nhyd}, where the closure relations are obtained by a
minimal energy constraint.

The non-hydrostatic model we intend to discretize in this paper  has
several interesting properties
\begin{itemize}
\item the model formulation only involves first order partial
  derivatives and appears as a depth-averaged version of the Euler system,
\item the proposed model is similar to the well-known Green-Naghdi
  model
but keeps a natural expression of the topography source term.
\end{itemize}

\subsection{The model}

So we start from the system (see
Fig.~\ref{fig:notations} for the notations)
\begin{align}
&\frac{\partial H}{\partial t} + \frac{\partial}{\partial x} \bigl(H\overline{u}\bigr) = 0, \label{eq:euler_11}\\
&\frac{\partial}{\partial t}(H\overline{u}) + \frac{\partial}{\partial x}\left(H \overline{u}^2
+ \frac{g}{2}H^2 + H\overline{p}_{nh}\right) = - (gH +
2 \overline{p}_{nh})\frac{\partial
  z_b}{\partial x},\label{eq:euler_22}\\
&\frac{\partial}{\partial t}(H\overline{w})  + \frac{\partial}{\partial
  x}(H\overline{w}\overline{u}) = 2
\overline{p}_{nh},\label{eq:euler_33}\\
& \frac{\partial (H\overline{u})}{\partial x} -
\overline{u} \frac{\partial (H+2z_b)}{\partial x} + 2\overline{w} = 0.
\label{eq:wbar1}
\end{align}
We consider this system for
$$t > t_0 \quad\mbox{and}\quad x \in [0,L],$$
${\bf \overline{u}} = (\overline{u},\overline{w})^T$ denotes the velocity vector and
$\overline{p}_{nh}$ the non-hydrostatic part of the pressure. The
total pressure is given by
\begin{equation}
\overline{p} =\frac{g}{2}H + \overline{p}_{nh}.
\label{eq:pres_tot}
\end{equation}
The quantities $(\overline{u},\overline{w},\overline{p})$
correspond to vertically averaged values of the variables $(u,w,p)$
arising in the incompressible Euler system.
\begin{figure}[htbp]
\begin{center}
\resizebox{9cm}{!}{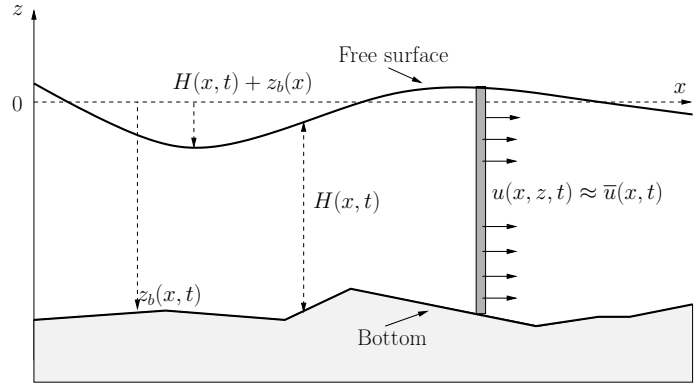}
\caption{Notations: water depth $H(x,t)$, free surface $H+z_b(x,t)$ and bottom $z_b(x,t)$.}
\label{fig:notations}
\end{center}
\end{figure}

The smooth solutions $H$, $\overline{u}$, $\overline{w}$, $\overline{p}_{nh}$
of the system~\eqref{eq:euler_11}-\eqref{eq:wbar1} also satisfy the energy balance
\begin{equation}
\frac{\partial}{\partial t}(\eta + gHz_b) + \frac{\partial}{\partial x}
\left(\overline{u}\bigl( \eta + gHz_b+\frac{g}{2}H^2 + H\overline{p}_{nh}\bigr)\right) = 0,
\label{eq:euler_55}
\end{equation}
where
\begin{equation}
\eta  =\frac{H(\overline{u}^2+\overline{w}^2)}{2}+ \frac{g}{2}H^2.
\label{eq:Ebar}
\end{equation}
We can rewrite~\eqref{eq:euler_55} under the form
\begin{equation}
\frac{\partial \tilde{\eta}}{\partial t} + \frac{\partial \widehat{G}}{\partial x}
= 0,
\label{eq:euler_66}
\end{equation}
with
\begin{equation}
\tilde{\eta} = \eta + gHz_b,\qquad G = \overline{u}\bigl(\eta+\frac{g}{2}H^2
\bigr),\qquad \tilde{G} =  G + gHz_b \overline{u},\qquad \widehat G = \widetilde G +
 H\overline{p}_{nh}\overline{u}.
\label{eq:def_entro}
\end{equation}
In the sequel, we will also use the definitions
\begin{eqnarray}
\eta_{hyd} & = & \frac{H}{2}\overline{u}^2 +
\frac{g}{2}(H)^2,\label{eq:eta_hyd}\\
\tilde{\eta}_{hyd} & = & \eta_{hyd} + gH z_b,\label{eq:etatilde}\\
G_{hyd} & = & \overline{u}
(\eta_{hyd} + \frac{g}{2}H^2), \label{eq:Ghyd}\\
\tilde{G}_{hyd} & = & \overline{u}
(\tilde{\eta}_{hyd} + \frac{g}{2}H^2), \label{eq:Gtilde}
\end{eqnarray}
which are functions of the unknowns.

The system~\eqref{eq:euler_11}-\eqref{eq:wbar1} is completed with initial and boundary
conditions that will be precised later.

\begin{remark}
Notice that simple manipulations of
Eqs.~\eqref{eq:euler_11} and~\eqref{eq:wbar1} lead to the relation
\begin{equation}
\frac{\partial}{\partial t}\left(\frac{H^2 + 2Hz_b}{2}\right)  +
\frac{\partial}{\partial x}\left(\frac{H^2 + 2Hz_b}{2}\overline{u}\right) = H \overline{w},\label{eq:euler_44}
\end{equation}
and Eq.~\eqref{eq:wbar1} could be replaced by~\eqref{eq:euler_44}. In this paper we
use~\eqref{eq:wbar1} which leads to keep the analogy with the divergence
operator in the Navier-Stokes equations as shown in the following paragraph.
\end{remark}

\subsection{A rewriting}
\label{subsec:3nd_writting}

Let us formally define the operator $\nablasw$
by
\begin{equation}
\nablasw f = \begin{pmatrix}
H\frac{\partial f}{\partial x} + \frac{\partial
  (H+2z_b)}{\partial x} f\\
-2 f
\end{pmatrix},
\label{eq:def_gradsw}
\end{equation}
that is a shallow water version of the gradient operator. Likewise, we
define a shallow water version of the divergence operator $\divsw$
under the form
\begin{equation}
\divsw {\bf u} =  \frac{\partial (Hu)}{\partial
  x} - u\frac{\partial (H+2z_b)}{\partial x} + 2 w,
\label{eq:def_divsw}
\end{equation}
where $ {\bf u}=(u,w)^T$. In definitions~\eqref{eq:def_gradsw}
and~\eqref{eq:def_divsw}, we assume the considered quantities are
smooth enough. The definition of the operators
$\nablasw$ and $\divsw$ implies we have the identity
\begin{equation}
\int_I \nablasw p . {\bf u}\ dx
= \left[ Hup \right]_{\partial I} -\int_I p  \divsw {\bf u}\ dx,\quad
\forall\ p,\ \forall {\bf u},
\label{eq:div_grad}
\end{equation}
where $I$ is any interval of $\R$. Notice that $\nablasw$ and $\divsw$
are $H$ and $z_b$ dependent operators and when necessary we will use
the notations $\nablasw(.;H)$ and $\divsw(.;H)$.

The system~\eqref{eq:euler_11}-\eqref{eq:wbar1} can be
rewritten under the compact form
\begin{align}
&\frac{\partial H}{\partial t} + \frac{\partial}{\partial x} \bigl(H\overline{u}\bigr) = 0,\label{eq:euler_sw1}\\
&\frac{\partial}{\partial t}(H\overline{\bf u}) +
\frac{\partial}{\partial x}\left( \overline{u} .
  H\overline{\bf u}\right) + \nabla_0 \left( \frac{g}{2}H^2 \right) +
\nablasw \overline{p}_{nh} =-gH\nabla_0 z_b,\label{eq:euler_sw2}\\
& \divsw \overline{\bf u} = 0,\label{eq:euler_sw3}
\end{align}
with the notation
$$\nabla_0 f = \begin{pmatrix} \frac{\partial f}{\partial x}\\ 0\end{pmatrix}.$$
Then the system~\eqref{eq:euler_sw1}-\eqref{eq:euler_sw3} appears as a
1d shallow water version of the 2d incompressible Euler system.

\subsection{Pressure equation}

For $H>0$, Eq.~\eqref{eq:euler_sw2} can also be written in the
nonconservative form
\begin{equation}
\frac{\partial \overline{\bf u}}{\partial t} +
\overline{u} \frac{\partial \overline{\bf u}}{\partial x} + g \nabla_0 H
+ \frac{1}{H}\nablasw \overline{p}_{nh} = -g \nabla_0 z_b,
\label{eq:momu}
\end{equation}
and applying~\eqref{eq:def_divsw} to Eq.~\eqref{eq:momu} leads
together with~\eqref{eq:euler_sw3} to the relation
\begin{multline}
-\frac{\partial}{\partial x} \left( H\frac{\partial
    \overline{p}_{nh}}{\partial x}\right) + \frac{1}{H}\left( 4 -
  H\frac{\partial^2 (H+2z_b)}{\partial x^2} +
  \left(\frac{\partial (H+2z_b)}{\partial x}\right)^2\right)
\overline{p}_{nh} = \\
2H \left(\frac{\partial
       \overline{u}}{\partial x}\right)^2 + 2 \overline{u}^2
   \frac{\partial^2 z_b}{\partial x^2} + gH \frac{\partial^2
     (H+z_b)}{\partial x^2} -2g \frac{\partial z_b}{\partial x}
   \frac{\partial (H+z_b)}{\partial x}.
\label{eq:sturm_liouville2}
\end{multline}
Notice that Eq.~\eqref{eq:sturm_liouville2} also reads
\begin{equation}
-\deltasw \overline{p}_{nh} = 2H \left(\frac{\partial
       \overline{u}}{\partial x}\right)^2 + 2 \overline{u}^2
   \frac{\partial^2 z_b}{\partial x^2} + gH \frac{\partial^2
     (H+z_b)}{\partial x^2} -2g \frac{\partial z_b}{\partial x}
   \frac{\partial (H+z_b)}{\partial x},
\label{eq:sturm_liouville3}
\end{equation}
with
$$\Delta_{sw} = \divsw \left(\frac{1}{H}\nablasw\right).$$
Conversely, Eq.~\eqref{eq:momu} with \eqref{eq:sturm_liouville3} give the divergence
free condition~\eqref{eq:euler_sw3}.
The resolution of Eq.~\eqref{eq:sturm_liouville3}~-- requiring the inversion
of a non local operator~-- gives the
expression for the non-hydrostatic pressure term $\overline{p}_{nh}$. A discrete
approximation of $\deltasw$ will be defined in paragraph~\ref{subsec:fully_discrete} and
used for the numerical solution of~\eqref{eq:euler_sw1}-\eqref{eq:euler_sw3}.


\begin{remark}
In all the writings of the model, the pressure term
$\overline{p}_{nh}$ appears as the Lagrange multiplier of the
divergence free condition. As in the incompressible Euler system, it
is not possible to derive {\it a priori} bounds for the
pressure terms. And hence, it is possible to obtain nonpositive values
for the total pressure $\overline{p}$ defined by~\eqref{eq:pres_tot}.

Such a situation means the fluid is no longer in contact with the
bottom and the formulation of the proposed model is no longer valid
since the bottom of the fluid has to be considered as a free surface.

Even if the proposed model can be modified to take into account these
situations, we do not consider them in this paper and we will propose in
paragraph~\ref{subsec:H0} a modification allowing to ensure that the
total pressure remains nonnegative.
\end{remark}

\subsection{Other formulations}

One of the most popular models for the description of long, dispersive
water waves is the Green-Naghdi model~\cite{green}. Several
derivations of the Green-Naghdi model have been proposed in the
literature~\cite{green,green1,su,miles}. For the mathematical justification of the model, the reader can refer
to~\cite{lannes,makarenko} and for its numerical approximation
to~\cite{lemetayer,chazel1,chazel2,JSM_CF}.

Introducing a parameter $\alpha$ and starting from the
system~\eqref{eq:euler_sw1}-\eqref{eq:euler_sw3}, we write
\begin{align}
&\frac{\partial H}{\partial t} + \frac{\partial}{\partial x} \bigl(H\overline{u}\bigr) = 0,\label{eq:gen_sw1}\\
&\frac{\partial}{\partial t}(H\overline{\bf u}) +
\frac{\partial}{\partial x}\left( \overline{u} .
  H\overline{\bf u}\right) + \nabla_0 \left( \frac{g}{2}H^2 \right) +
\nablaswa \overline{p}_{nh} = -gH\nabla_0 z_b,\label{eq:gen_sw2}\\
& \divswa \overline{\bf u} = 0,\label{eq:gen_sw3}
\end{align}
with
$$\nablaswa f = \begin{pmatrix}
H\frac{\partial f}{\partial x} + \frac{\partial
  (H+2z_b)}{\partial x} f\\
-\alpha f
\end{pmatrix},$$
and
\begin{equation}
\divswa {\bf u} = \frac{\partial (Hu)}{\partial
  x} - u\frac{\partial (H+2z_b)}{\partial x} + \alpha w.
\label{eq:gen_div}
\end{equation}
The value $\alpha=2$ gives exactly the
model~\eqref{eq:euler_11}-\eqref{eq:wbar1}. The system~\eqref{eq:gen_sw1}-\eqref{eq:gen_sw3} is completed with the energy balance
\begin{equation}
\frac{\partial \tilde{\eta}^\alpha}{\partial t} + \frac{\partial}{\partial x}
\left(\overline{u}\bigl(\tilde{\eta}^\alpha + \frac{g}{2}H^2 + H\overline{p}_{nh}\bigr)\right) = 0,
\label{eq:gen_sw4}
\end{equation}
where
\begin{equation}
\tilde{\eta}^\alpha = \frac{H}{2}\left(\overline{u}^2 +
  \frac{2\alpha-1}{3}\overline{w}^2 \right) + \frac{g}{2}H^2 + gHz_b.
\label{eq:gen_sw5}
\end{equation} 

Following~\cite{lemetayer} (see also~\cite{JSM_nhyd}), the
Green-Naghdi model with flat bottom reads
\begin{eqnarray}
& & \frac{\partial H}{\partial t} + \frac{\partial}{\partial x} \bigl(H\overline{u}\bigr) = 0,\label{eq:gn_11}\\
& & \frac{\partial (H\overline{u})}{\partial t} +
\frac{\partial}{\partial x} \left( H\overline{u}^2 +
  \frac{g}{2}H^2 + H \overline{p}_{gn}\right) = 0
,\label{eq:gn_22}\\
& &
  \frac{\partial}{\partial t} (H \overline{w}) +
  \frac{\partial }{\partial x} (H\overline{u}\overline{w})=
  \frac{3}{2}\overline{p}_{gn}\label{eq:gn_33},\\
& & \frac{\partial (H\overline{u})}{\partial
  x} - \overline{u}\frac{\partial H}{\partial x} + \frac{3}{2}
\overline{w} = 0, \label{eq:gn_44}
\end{eqnarray}
corresponding to~\eqref{eq:gen_sw1}-\eqref{eq:gen_sw3} with the value
$\alpha=3/2$. And hence it appears that the proposed model and the Green-Naghdi system
only differ from the value of $\alpha$ in Eqs.~\eqref{eq:gen_sw1},\eqref{eq:gen_sw3}.

Notice that the fundamental duality relation
$$\int_I \overline{p}_{nh}\divswa \overline{\bf u}\  dx = \left[ H\overline{u}\overline{p}_{nh} \right]_{\partial I} -\int_I \nablaswa \overline{p}_{nh}. \overline{\bf u}\ dx,$$
holds for any interval $I$.

It seems to the authors that the consistency of relations
Eqs.~\eqref{eq:gen_div} and~\eqref{eq:gen_sw5} with the divergence
free condition and the energy balance for the 2d Euler system is only obtained for the value $\alpha=2$ and in this paper we mainly focus on the choice $\alpha=2$ i.e. on the
model~\eqref{eq:euler_sw1}-\eqref{eq:euler_sw3}. But the numerical
procedure proposed and described in this paper is also valid for the model corresponding to any
nonnegative value of $\alpha$.

\section{Kinetic description}
\label{sec:kin}

In this section, we propose a kinetic interpretation for the system
\eqref{eq:euler_11}-\eqref{eq:euler_33} completed
with~\eqref{eq:euler_55}. The kinetic description will be used in
Section~\ref{sec:num} to derive a stable, accurate and robust
numerical scheme.

The kinetic approach consists in using a description of the microscopic
 behavior of the system \cite{perthame}. 
In this method, a fictitious density of particles is introduced and the equations
are considered at the microscopic scale, where no discontinuity
occurs. The kinetic interpretation of a system allows its transformation
into a family of linear transport equations, to which an upwinding discretization
is naturally applicable.

Following~\cite{perthame}, we introduce a real function $\chi$ defined on
$\mathbb{R}$, compactly supported and which has the following properties
\begin{equation}
\left\{\begin{array}{l}
\chi(-w) = \chi(w) \geq 0\\
\int_{\mathbb{R}} \chi(w)\ dw = \int_{\mathbb{R}} w^2\chi(w)\ dw = 1.
\end{array}\right.
\label{eq:chi1}
\end{equation}
Among all the functions $\chi$ satisfying~\eqref{eq:chi1}, one plays
an important role. Indeed, the choice
\begin{equation}
\chi(z) = \frac{1}{\pi} \left(1 - \frac{z^2}{4}\right)_+^{1/2},
\label{eq:chi0}
\end{equation}
with $x_+ \equiv \max(0,x)$, allows to ensure important stability
properties~\cite{simeoni,JSM_entro}. In the following, we keep this
special choice for $\chi$.

\subsection{Kinetic interpretation of the Saint-Venant system}

The classical Saint-Venant system~\cite{saint-venant,gerbeau}
corresponds to the hydrostatic part of the model~\eqref{eq:euler_11}-\eqref{eq:euler_33}, it reads
\begin{align}
&\frac{\partial H}{\partial t} + \frac{\partial}{\partial x} \bigl(H\overline{u}\bigr) = 0, \label{eq:sv_1}\\
&\frac{\partial}{\partial t}(H\overline{u}) + \frac{\partial}{\partial x}\left(H \overline{u}^2
+ \frac{g}{2}H^2 \right) = - gH \frac{\partial
  z_b}{\partial x},\label{eq:sv_2}
\end{align}
completed with the entropy inequality
\begin{eqnarray}
\frac{\partial \tilde{\eta}_{hyd}}{\partial t} + \frac{\partial \tilde{G}_{hyd}}{\partial x} \leq 0,\label{eq:sv_3}
\end{eqnarray}
with $\tilde{\eta}_{hyd}$ and~$\tilde{G}_{hyd}$ defined by~\eqref{eq:etatilde},\eqref{eq:Gtilde}.

Let us construct the density of particles $M(x,t,\xi)$ playing the
role of a Maxwellian: the microscopic density of particles present at
time $t$, at the abscissa $x$ and with
velocity $\xi$ given by
\begin{equation}
M(H,\overline{u},\xi) = \frac{H}{c} \chi\left(\frac{\xi - \overline{u}}{c}\right)
= \frac{1}{g\pi}\Bigl(2gH-(\xi-\overline{u})^2\Bigr)_+^{1/2},
\label{eq:kinmaxw}
\end{equation}
with $c = \sqrt{\frac{gH}{2}}$, $\xi\in\R$. The equilibrium defined by~\eqref{eq:kinmaxw}
corresponds to the classical kinetic Maxwellian equilibrium, used in \cite{simeoni} for example.
It satisfies the following moment relations,
\begin{equation}\begin{array}{c}
	\dsp \int_\R \kxi M(H,\overline{u},\xi)\,d\xi= \left(\begin{array}{c}
          H\\ H \overline{u} \end{array}\right),\\
	\dsp\int_\R \xi^2 M(H,\overline{u},\xi)\,d\xi=H \overline{u}^2+g\frac{H^2}{2}.
	\label{eq:kinmom}
	\end{array}
\end{equation}
The interest of the particular form \eqref{eq:kinmaxw} lies in its
link with a kinetic entropy, see~\cite{JSM_entro} where the properties of
$H_K(f,\xi,z)$ are studied, $H_K$ refers to the kinetic entropy used in~\cite{JSM_entro}.
Consider the kinetic entropy,
\begin{equation}
	H_K(f,\xi,z)=\frac{\xi^2}{2}f+\frac{g^2\pi^2}{6}f^3+gzf,
	\label{eq:kinH}
\end{equation}
where $f\geq 0$, $\xi\in\R$ and $z\in\R$, and its version without topography
\begin{equation}
	H_{K,0}(f,\xi)=\frac{\xi^2}{2}f+\frac{g^2\pi^2}{6}f^3.
	\label{eq:kinH0}
\end{equation}
Then one can check the relations
\begin{equation}
	\int_\R H_K\bigl(M(H,\overline{u},\xi),\xi,z_b\bigr)\,d\xi= \tilde{\eta}_{hyd},
	\label{eq:kinint}
\end{equation}
\begin{equation}
	\int_\R \xi H_K\bigl(M(H,\overline{u},\xi),\xi,z_b\bigr)\,d\xi=
        \tilde{G}_{hyd}.
	\label{eq:kinflint}
\end{equation}

These definitions allow us to obtain a kinetic representation of the
Saint-Venant system~\cite{simeoni}.
\begin{proposition}
The pair of functions $(H, H\overline{u})$ is a strong solution of the
Saint-Venant system~\eqref{eq:sv_1}-\eqref{eq:sv_2} if and only if $M(H,\overline{u},\xi)$ satisfies
the kinetic equation
\begin{equation}
({\cal B})\qquad \frac{\partial M}{\partial t} + \xi \frac{\partial M}{\partial x} -
g\frac{\partial z_b}{\partial x}\frac{\partial M}{\partial \xi}  = Q,
\label{eq:gibbs}
\end{equation}
for some ``collision term'' $Q(x,t,\xi)$ which satisfies, for a.e. $(x,t)$, 
\begin{equation}
\int_{\mathbb{R}}  Q\ d\xi = \int_{\mathbb{R}}  \xi Q\ d\xi = 0.
\label{eq:collision}
\end{equation}
\label{prop:kinetic_sv}
\end{proposition}
\begin{proof}[Proof of prop.~\ref{prop:kinetic_sv}]
Using~\eqref{eq:kinmom}, the proof relies on a very obvious computation.
\end{proof}
\begin{remark}
The proposition~\ref{prop:kinetic_sv} remains valid if instead
of~\eqref{eq:chi0}, the equilibrium $M$ is built with any function satisfying~\eqref{eq:chi1}.
\end{remark}
This proposition produces a very useful consequence. The non-linear
shallow water system can be viewed as a single linear equation for a
scalar function $M$ depending nonlinearly on $H$ and $\overline{u}$, for which it is easier to find simple
numerical schemes with good theoretical properties.

\subsection{Kinetic interpretation of the depth-averaged Euler system}

Since we take into account the non-hydrostatic effects of the pressure,
the microscopic vertical velocity $\gamma$ of the particles has to be considered and
we now construct the new density of particles $M(x,t,\xi,\gamma)$
defined by a Gibbs equilibrium: the microscopic density of particles present at
time $t$, abscissa $x$ and with
microscopic horizontal velocity $\xi$ and microscopic vertical velocity $\gamma$ is given by
\begin{equation}
M(x,t,\xi,\gamma) = \frac{H}{c} \chi\left(\frac{\xi -
\overline{u}}{c}\right)\delta\left(\gamma -
\overline{w}\right),
\label{eq:Mnh}
\end{equation}
where $\delta$ is the Dirac distribution and $c=\sqrt{\frac{gH}{2}}$.

Then we have the following proposition.
\begin{proposition}
For a given $\overline{p}_{nh}$, the functions
$(H,\overline{u},\overline{w})$ satisfying the divergence free
condition~\eqref{eq:def_divsw}, are strong solutions of
 the depth-averaged Euler system described in
 \eqref{eq:euler_11}-\eqref{eq:euler_33},\eqref{eq:euler_55} if and only if the
 equilibrium $M(x,t,\xi,\gamma)$ is solution of the kinetic equations
\begin{multline}
({\cal B}_{nh}) \qquad \frac{\partial M}{\partial t} +
\xi\frac{\partial M}{\partial x} - \left(
  \left( g+\frac{2\overline{p}_{nh}}{H} \right) \frac{\partial
    z_b}{\partial x} + \frac{1}{H} \frac{\partial}{\partial x}(H\overline{p}_{nh})\right)\frac{\partial
M}{\partial \xi} \\
+ \frac{2\overline{p}_{nh}}{H}
\frac{\partial M}{\partial \gamma}
= Q_{nh},\label{eq:gibbs+_1}
\end{multline}
where $Q_{nh}=Q_{nh}(x,t,\xi,\gamma)$ is  a ``collision term'' satisfying
\begin{eqnarray}
& & \int_{\mathbb{R}^2} Q_{nh}\ d\xi d\gamma = \int_{\mathbb{R}^2} \xi Q_{nh}\ d\xi d\gamma = \int_{\mathbb{R}^2} \gamma Q_{nh}  d\xi d\gamma = 0. \label{eq:collision1}
\end{eqnarray}
Additionally, the solution is an entropy solution if
\begin{equation}
\int_{\mathbb{R}^2} \left(\frac{\xi^2 + \gamma^2}{2} + \frac{g^2\pi^2}{2}M^2 + gz_b\right) Q_{nh} d\gamma d\xi  \leq 0.
\label{eq:calc_kin_av}
\end{equation}
\label{prop:euler_av_kinetic}
\end{proposition}
Notice that in the case of the Saint-Venant system, the particular
choice of $M$ defined by~\eqref{eq:kinmaxw} ensures
$$\int_{\mathbb{R}} \left(\frac{\xi^2}{2} + \frac{g^2\pi^2}{2}M^2 + gz_b\right) Q d\xi  = 0.$$
\begin{proof}[Proof of prop.~\ref{prop:euler_av_kinetic}]
From the definitions \eqref{eq:chi1},\eqref{eq:Mnh} and
\eqref{eq:collision1}, the proof results from easy computations,
namely by integrating the relation~\eqref{eq:gibbs+_1}
$$\int_{\R^2} ({\cal B}_{nh})\ d\xi d\gamma, \quad \int_{\R^2} \xi ({\cal B}_{nh})\
d\xi d\gamma,\quad\mbox{and}\quad \int_{\R^2} \gamma ({\cal B}_{nh})\
d\xi d\gamma.$$
Likewise, the energy balance is obtained calculating the quantity
$$\int_{\R^2} \left( \frac{\xi^2+\gamma^2}{2} + \frac{g^2\pi^2}{2}M^2 + gz_b\right)({\cal B}_{nh})\ d\xi d\gamma.$$
\end{proof}
Because of the special role of the equation~\eqref{eq:wbar1}, it is
not easy to describe it at the kinetic level.

\section{Numerical scheme}
\label{sec:num}

In this section we propose a discretization for the
system~\eqref{eq:euler_sw1}-\eqref{eq:euler_sw3}. In order to process
step by step, we first establish some properties for the
semi-discrete schemes in time and then in space. Then we study the fully
discrete scheme.

For the sake of simplicity, the notations with~~$\bar{}$~ are dropped. We write the system~\eqref{eq:euler_sw1}-\eqref{eq:euler_sw2} in a condensed form
\begin{equation}
\frac{\partial X}{\partial t} + \frac{\partial}{\partial x}F(X) +
R_{nh}= S(X),
\label{eq:form_fin}
\end{equation}
with
\begin{eqnarray}
X = \left(\begin{array}{c} H\\ Hu\\ Hw\end{array}\right),\quad
F(X) = \left(\begin{array}{c} Hu\\ Hu^2 +
    \frac{g}{2}H^2\\ Huw
\end{array}\right),\quad S(X) =
\left(\begin{array}{c} 0\\ -gH\nabla_0 z_b\end{array}\right),
\label{eq:X}
\end{eqnarray}
and
$$R_{nh} = \left(\begin{array}{c} 0\\ \nablasw
    p_{nh} \end{array}\right),$$
with $\nablasw p_{nh}$ defined by~\eqref{eq:def_gradsw}. The
expression of $p_{nh}$ is defined by~\eqref{eq:sturm_liouville3} and ensures the
divergence free condition~\eqref{eq:euler_sw3} is satisfied.

\subsection{Fractional step scheme}
\label{subsec:gen_scheme}

For the time discretization, we denote $t^n = \sum_{k \leq n} \Delta t^k$ where the time steps $\Delta t^k$ will be
precised later though a CFL condition. Following~\cite{chorin}, we use
an operator splitting technique resulting in a two step scheme
\begin{eqnarray}
& & \frac{X^{n+1/2} - X^n}{\Delta t^n} + \frac{\partial}{\partial x}F(X^n)
  = S(X^n),\label{eq:form_fin1}\\
& & \frac{X^{n+1} - X^{n+1/2}}{\Delta t^n} + R_{nh}^{n+1}
  = 0.\label{eq:form_fin2}
\end{eqnarray}
The non-hydrostatic part of the pressure $p_{nh}^{n+1}$ is defined
by~\eqref{eq:sturm_liouville3} and ensures, as already said, that the divergence
free constraint~\eqref{eq:euler_sw3} is satisfied i.e.
\begin{equation}
\divsw {\bf u}^{n+1} = 0.
\label{eq:euler_sw3_d}
\end{equation}
The discretization of Eq.~\eqref{eq:sturm_liouville3} is given
hereafter. The system~\eqref{eq:form_fin1}-\eqref{eq:form_fin2} has to be
completed with suitable boundary conditions that will be precised
later, see paragraph~\ref{subsec:BC}.

The prediction step~\eqref{eq:form_fin1} consists in the resolution of the
Saint-Venant system and a transport equation for $(Hw)^{n+1/2}$ i.e.
\begin{eqnarray}
& & H^{n+1/2} = H^n - \Delta t^n \frac{\partial (Hu)^n}{\partial x},\label{eq:semi_dis1}\\
& & (Hu)^{n+1/2} = (Hu)^n - \Delta t^n \frac{\partial }{\partial x}
\left(Hu^2 + \frac{g}{2}H^2\right)^n - \Delta t^n gH^n \frac{\partial z_b}{\partial x},\label{eq:semi_dis2}\\
& & (Hw)^{n+1/2} = (Hw)^n - \Delta t^n \frac{\partial (Hwu)^n}{\partial x},\label{eq:semi_dis3}
\end{eqnarray}
and the correction step~\eqref{eq:form_fin2} writes
\begin{eqnarray}
& & H^{n+1} = H^{n+1/2},\label{eq:semi_disH}\\
& &  {\bf u}^{n+1} = {\bf u}^{n+1/2} - \frac{\Delta t^n}{H^{n+1}} \nablasw p_{nh}^{n+1},\label{eq:semi_dis4}
\end{eqnarray}
with
$${\bf u}^{n+1}=\biggl(
  \frac{(Hu)^{n+1}}{H^{n+1}},\frac{(Hw)^{n+1}}{H^{n+1}}\biggr)^T.$$
More precisely, due to the expression of the operator $\nablasw$ given
in~\eqref{eq:def_gradsw}, the
notations $\nablasw p_{nh}^{n+1}$ means $\nablasw
(p_{nh}^{n+1};H^{n+1})$ and the same remark holds for the operator $\divsw$. 
Then inserting ${\bf u}^{n+1}=(u ^{n+1},w ^{n+1})$ satisfying~\eqref{eq:semi_dis4} in
relation~\eqref{eq:euler_sw3_d} gives the governing equation for $p_{nh}^{n+1}$
\begin{equation}
\divsw \left(\frac{1}{H^{n+1}}\nablasw p_{nh}^{n+1} \right) = \frac{1}{\Delta t^n} \divsw
\biggl(
  \frac{(Hu)^{n+1/2}}{H^{n+1/2}},\frac{(Hw)^{n+1/2}}{H^{n+1/2}}\biggr)^T, 
\label{eq:elliptic}
\end{equation}
that is a discrete version of~\eqref{eq:sturm_liouville3}. 
Notice also that in Eq.~\eqref{eq:semi_dis4} we have used the fact that
$H^{n+1} = H^{n+1/2}$. It appears that the right hand side
of~\eqref{eq:elliptic} can be evaluated by the conservative
variables $(H,Hu,Hw)^{n+1/2}$ given by~\eqref{eq:semi_dis1}-\eqref{eq:semi_dis3} , the first
step of the time scheme.

\begin{proposition}
The scheme~\eqref{eq:form_fin1}-\eqref{eq:euler_sw3_d}
satisfies a semi-discrete (in time) entropy inequality of the
form
$$\tilde{\eta}^{n+1} \leq \tilde{\eta}^n -\Delta t^n
\frac{\partial}{\partial x}\left( \tilde{G}^n +
  (Hu)^{n+1}p_{nh}^{n+1} \right) + (\Delta t^n)^2 {\cal O}\biggl(\left\|
F'(X^n)+S(X^n)\right\|^2_2\biggr),$$
with
$$\tilde{\eta}^n = \tilde{\eta}(X^n) = \frac{H^n}{2} \Bigl( (u^n)^2 +  (w^n)^2 \Bigr) +
\frac{g}{2} (H^n)^2 + gH^nz_b,$$
defined by~\eqref{eq:def_entro}.
\label{prop:stability2}
\end{proposition}

\begin{proof}[Proof of prop.~\ref{prop:stability2}]
Multiplying Eq.~\eqref{eq:semi_dis2} 
by $u^n$, we obtain after classical manipulations
\begin{multline}
\tilde{\eta}_{hyd}^{n+1/2}= \tilde{\eta}_{hyd}^n - \Delta t^n \frac{\partial }{\partial x}
\left( u^n \left(\tilde{\eta}_{hyd}^n + \frac{g}{2}(H^n)^2\right)\right) +
\frac{g}{2} (H^{n+1/2} - H^n)^2 \\
+ \frac{H^{n+1/2}}{2} (u^{n+1/2} -
u^n)^2,
\label{eq:error_sdt_1}
\end{multline}
with
$$\tilde{\eta}_{hyd}^n = \tilde{\eta}_{hyd}(X^n) = \frac{H^n}{2}(u^n)^2 +
\frac{g}{2}(H^n)^2 + gH^n z_b,$$
defined by~\eqref{eq:etatilde}. Likewise, multiplying~Eq.~\eqref{eq:semi_dis3} by $w^n$ leads to
\begin{multline}
\frac{H^{n+1/2}}{2}(w^{n+1/2})^2 = \frac{H^n}{2}(w^{n})^2 - \Delta t^n \frac{\partial }{\partial x}
\left( u^n \frac{H^n}{2}(w^n)^2\right) \\
+
\frac{H^{n+1/2}}{2} (w^{n+1/2} - w^n)^2.
\label{eq:error_sdt_2}
\end{multline}
Notice that the last term appearing in Eq.~\eqref{eq:error_sdt_1} and
in Eq.~\eqref{eq:error_sdt_2}
is non negative. These error terms are due to the explicit time scheme. The sum of the two previous equations gives the inequality
\begin{equation}
\tilde{\eta}^{n+1/2} \leq \tilde{\eta}^n -\Delta t^n \frac{\partial \tilde{G}^n}{\partial x} + (\Delta t^n)^2 {\cal O}\biggl(\left\|
F'(X^n)+S(X^n)\right\|^2_2\biggr).
\label{eq:energ_semidt1}
\end{equation}
Now we multiply~\eqref{eq:semi_dis4} by $(H{\bf u})^{n+1}$ and
after simple manipulations it comes
\begin{multline}
\frac{H^{n+1}}{2}(u^{n+1})^2=
\frac{H^{n+1/2}}{2}(u^{n+1/2})^2 - \Delta t^n \left(
  \frac{\partial}{\partial x} \bigl( (Hu)^{n+1}p_{nh}^{n+1} \bigr) \right.\\
\left. + p_{nh}^{n+1} \left(
    \frac{\partial}{\partial x} (Hu)^{n+1} - u^{n+1} \frac{\partial}{\partial x}(H^{n+1}+2z_b)\right)\right)-
\frac{H^{n+1/2}}{2} (u^{n+1} - u^{n+1/2})^2.
\label{eq:energ_semidt3}
\end{multline}
and
\begin{multline}
\frac{H^{n+1}}{2}(w^{n+1})^2=
\frac{H^{n+1/2}}{2}(w^{n+1/2})^2 + 2 \Delta t^n p_{nh}^{n+1}w^{n+1} -
\frac{H^{n+1/2}}{2} (w^{n+1} - w^{n+1/2})^2.
\label{eq:energ_semidt4}
\end{multline}
In Eqs.~\eqref{eq:energ_semidt3} and~\eqref{eq:energ_semidt4}, the error terms due to the time discretization are non-positive. Using the two previous equations and~\eqref{eq:euler_sw3_d} gives the inequality
\begin{equation}
\tilde{\eta}^{n+1} \leq \tilde{\eta}^{n+1/2} -\Delta t^n
\frac{\partial }{\partial x} \left(
(Hu)^{n+1}p_{nh}^{n+1}\right).
\label{eq:energ_semidt2}
\end{equation}
Finally Eq.~\eqref{eq:energ_semidt2} coupled with
Eq.~\eqref{eq:energ_semidt1} gives the result.
\end{proof}

\subsection{The semi-discrete (in space) scheme}
\label{subsec:semi_d_space}

To approximate the solution $X=(H,Hu,Hw)^T$ of the system~\eqref{eq:form_fin}, we use a finite volume framework.
We assume that the computational domain is discretized with $I$ nodes
$x_i$, $i=1,\ldots,I$.
  We denote $C_i$ the cell of length $\Delta
x_i=x_{i+1/2}-x_{i-1/2}$ with $x_{i+1/2}=(x_i+x_{i+1})/2$. We denote
$X_{i}=(H_i,q_{x,i},q_{z,i})^T$ with
$$X_{i} \approx \frac{1}{\Delta x_i} \int_{C_i} X(x,t) dx,$$
the approximate solution at time $t$ on
the cell $C_i$ with $q_{x,i}=H_i u_{i}$, $q_{z,i}=H_i w_{i}$. Likewise for the
topography, we define
$$z_{b,i} = \frac{1}{\Delta x_i} \int_{C_i} z_b(x) dx.$$
The non-hydrostatic part of the pressure is discretized on a staggered
grid (in fact the dual
mesh if we consider the 2d case)
$$p_{nh,i+1/2} \approx \frac{1}{\Delta x_{i+1/2}} \int_{x_i}^{x_{i+1}} p_{nh}(x,t) dx,$$
$\Delta x_{i+1/2} = x_{i+1} - x_i$.

Now we propose and study the semi-discrete (in space) scheme
approximating the model~\eqref{eq:form_fin} and the divergence free condition~\eqref{eq:euler_sw3}. The semi-discrete scheme writes
\begin{eqnarray}
&& \Delta x_i \frac{\partial X_i}{\partial t} + \left( F_{i+1/2-} -
  F_{i-1/2+}\right) + R_{nh,i} = 0,\label{eq:mom_d}\\
&& \divswd (\{{\bf u}_j\}) = 0,\label{eq:div_d}
\end{eqnarray}
where~\eqref{eq:div_d} is a discretized version of the divergence free
condition~\eqref{eq:euler_sw3} which we detail below and
with the numerical fluxes
\begin{align*}
F_{i+1/2+} & = {\cal F}(X_{i},X_{i+1},z_{b,i},z_{b,i+1}) + {\cal S}_{i+1/2+}\\
F_{i+1/2-} & = {\cal F}(X_{i},X_{i+1},z_{b,i},z_{b,i+1}) + {\cal S}_{i+1/2-}.
\end{align*}
${\cal F}$ is a numerical flux for the conservative part of the system,
${\cal S}$ is a convenient discretization of the topography source
term, see paragraph~\ref{subsec:fully_discrete}.

Since the first two lines
of~\eqref{eq:form_fin1} correspond to the classical Saint-Venant
system, the numerical fluxes
\begin{equation}
F_{i+1/2\pm} = \left(\begin{array}{c}
F_{H,i+1/2}\\
F_{q_x,i+1/2\pm}\\
F_{q_z,i+1/2}
\end{array}\right),
\label{eq:flux_hyp}
\end{equation}
can be constructed using any numerical solver
for the Saint-Venant system. More precisely for $F_{H,i+1/2}$,$F_{q_x,i+1/2\pm}$ we adopt numerical fluxes suitable for the
Saint-Venant system with topography. Notice that from the
definition~\eqref{eq:X}, since only the second
component of $S(X)$ is non zero, only $F_{q_x}$ has two interface values
under the form $F_{q_x,i+1/2\pm}$. For the definition of $F_{q_z,i+1/2}$, the formula (see~\cite{bristeau4})
\begin{equation}
F_{q_z,i+1/2} = F_{H,i+1/2}w_{i+1/2},
\label{eq:fluxHw}
\end{equation}
with
\begin{equation}
w_{i+1/2} = \left\{ \begin{array}{ll}
w_i & \mbox{ if } F_{H,i+1/2} \geq 0\\
w_{i+1} & \mbox{ if } F_{H,i+1/2} < 0
\end{array}\right.
\label{eq:def_upwind_w}
\end{equation}
can be used.

Combining the finite volume approach for the hyperbolic part with a
finite difference strategy for the parabolic part, the non-hydrostatic part $R_{nh,i}$ is defined by
$$R_{nh,i} = \begin{pmatrix} 0\\ \nablaswd p_{nh}\end{pmatrix},$$
where the two components of $\nablaswd p_{nh}$ are defined by
\begin{eqnarray}
 \Delta x_i \left. \nablaswd p_{nh}\right|_1 & = & 
H_{i}
  (p_{nh,i+1/2} - p_{nh,i-1/2}) \nonumber\\
& & + p_{nh,i+1/2} \bigl(\zeta_{i+1} - \zeta_{i}\bigr) + p_{nh,i-1/2} \bigl(\zeta_{i}  - \zeta_{i-1}\bigr),\label{eq:pnh_sd1}\\
 \Delta x_i \left. \nablaswd p_{nh}\right|_2 & = & 
-\Bigl( \Delta x_{i+1/2} p_{nh,i+1/2} + \Delta x_{i-1/2} p_{nh,i-1/2}\Bigr),\label{eq:pnh_sd2}
\end{eqnarray}
with
\begin{eqnarray*}
& & \zeta_{i} = \frac{H_{i}+2z_{b,i}}{2}.
\end{eqnarray*}
And in~\eqref{eq:div_d},  $\divswd ({\bf u})$ is defined by
\begin{multline}
\Delta x_{i+1/2} \divswd ({\bf u}) = (Hu)_{i+1} - (Hu)_{i} - (u_i + u_{i+1}) \bigl( \zeta_{i+1} -
  \zeta_{i}\bigr) \\
+ \Delta x_{i+1/2} \bigl(w_{i+1} + w_i\bigr).
\label{eq:div_ip12}
\end{multline}
Notice that in the definitions~\eqref{eq:pnh_sd1}-\eqref{eq:pnh_sd2} and in the sequel,
the quantity $p_{nh}$ means $\{p_{nh,j}\}$. Likewise in
Eq.~\eqref{eq:div_ip12} and in the sequel, ${\bf u}$ means $\{{\bf u}_j\}$.

In a first step, we assume we have for the resolution of the
hyperbolic part i.e. the calculus of $F_{i+1/2-}$,$F_{i-1/2+}$, a robust and efficient numerical scheme. Since this
step mainly consists in the resolution of the
Saint-Venant equations there exists several solvers endowed with such
properties (HLL, Rusanov, relaxation, kinetic,\ldots),
see~\cite{bouchut_book}.

We assume we have for the prediction step a numerical scheme which is
\begin{itemize}
\item[{\it (i)}] consistent with the Saint-Venant
  system~\eqref{eq:sv_1}-\eqref{eq:sv_2},
\item[{\it (ii)}] well-balanced i.e. at rest $X^{n+1/2}=X^n$ in~\eqref{eq:form_fin1},
\item[{\it (iii)}] satisfying an in-cell entropy of the form
$$\Delta x_i \frac{\partial \tilde{\eta}_{hyd,i}}{\partial t}  + \Bigl( \tilde{G}_{hyd,i+1/2}
- \tilde{G}_{hyd,i-1/2} \Bigr) \leq 0,$$
with
$$\tilde{\eta}_{hyd,i}=\frac{H_i}{2}(u_i)^2 +
\frac{g}{2}(H_i)^2 + gH_i z_b,$$
and $\tilde{G}_{hyd,i+1/2}$ is the entropy flux associated with the
chosen finite volume solver.
\end{itemize}
Then the following proposition holds.
\begin{proposition}
The numerical scheme~\eqref{eq:mom_d},\eqref{eq:div_d} 
\begin{itemize}
\item[{\it (i)}] is consistent with the
  model~\eqref{eq:euler_11}-\eqref{eq:wbar1},
\item[{\it (ii)}] preserves the same steady state as the lake at rest,
\item[{\it (iii)}] satisfies an in-cell entropy inequality associated
  with the entropy $\tilde{\eta}(t)$ analogous to the continuous one
  defined in~\eqref{eq:euler_66}
\begin{equation}
\Delta x_i \frac{\partial \tilde{\eta}_i}{\partial t}  + \Bigl( \hat{G}_{i+1/2}
- \hat{G}_{i-1/2} \Bigr) \leq d_i, \quad\mbox{in } C_i,
\label{eq:semi_dis_entro}
\end{equation}
with
\begin{eqnarray*}
\tilde{\eta}_{i} & = & \tilde{\eta}(X_{i}) = \tilde{\eta}_{hyd,i} +
H_{i} \frac{w_{i}^2}{2},\\ 
\hat{G}_{i+1/2} & = & \tilde{G}_{hyd,i+1/2} +
F_{H,i+1/2}w_{i+1/2}^2/2 + (Hu)_{i+1/2} p_{nh,i+1/2}.
\end{eqnarray*}
 and
$d_i$ is an error term satisfying $d_i = {\cal O} (\Delta x^3)$,
\item[{\it (iv)}] ensures a decrease of the total energy under the form
\begin{equation}
\frac{\partial}{\partial t} \sum_i \Delta x_i \tilde{\eta}_i  \leq 0.
\label{eq:semi_dis_entro1}
\end{equation}
\end{itemize}
\label{prop:entropy_semi_d}
\end{proposition}
The inequality~\eqref{eq:semi_dis_entro} is obtained by
multiplying (scalar product) the two momenta equations of~\eqref{eq:mom_d} by
${\bf u}_i$ which corresponds to a piecewise
constant discretization. For the hyperbolic part it allows to derive a semi-discrete entropy, see~\cite{bristeau1,JSM_entro}.

The error term $d_i$ in the r.h.s. of~\eqref{eq:semi_dis_entro} comes
from the discretization of the non-hydrostatic part corresponding to the incompressible part
of the model.
In order to eliminate $d_i$, a more
acurate discretization of the velocity field~-- in accordance with the
approximation of the divergence free condition~-- would be necessary.

\begin{proof}[Proof of prop.~\ref{prop:entropy_semi_d}]
{\it (i)} Since we have assumed that the numerical scheme for the
prediction part is consistent, we have
$${\cal F}(X,X,z,z) = F(X).$$
Likewise ${\cal S}$ is a consistent discretization of the topography
source term. It is easy to prove the non-hydrostatic terms given
by~\eqref{eq:pnh_sd1},\eqref{eq:pnh_sd2} is a consistent
discretization of $R_{nh}$ that proves the result.

{\it (ii)} When ${\bf u}_j=(0,0)^T$ for $j=i-1,i,i+1$, the hyperbolic
part being discretized using a well-balanced scheme we have
$$F_{i+1/2-}=F_{i-1/2+}=0,$$
and the scheme~\eqref{eq:mom_d},\eqref{eq:div_d}
reduces to
$$R_{nh,i}=(0,0,0)^T,\qquad \frac{\partial X_i}{\partial t} = 0,$$
ensuring the scheme is well-balanced.

{\it (iii)} Multiplying the first two equations of
system~\eqref{eq:mom_d} by the first two components of $\tilde{\eta}'(X_i)$ with
$$\tilde{\eta}'(X_i) = \begin{pmatrix}
gH_i - \frac{u_i^2+w_i^2}{2}\\
u_i\\
w_i
\end{pmatrix},$$
we obtain
\begin{equation}
\Delta x_i \frac{\partial \tilde{\eta}_{hyd,i}}{\partial t} + \left( \tilde{G}_{hyd,i+1/2-} -
  \tilde{G}_{hyd,i-1/2+}\right) +
u_i \left. \nablaswd p_{nh} \right|_1 \leq 0.
\label{eq:energ_d}
\end{equation}
In Eq.~\eqref{eq:energ_d}, the three first terms are obtained as
in~\cite{bristeau1}. The proof of theorem 2.1 in~\cite{bristeau1} can
be used without any change, except for the vertical kinetic energy,
namely
$$H\frac{w^2}{2},$$
that is not considered in~\cite{bristeau1} since the model is
hydrostatic. In order to obtain the contribution of the vertical
kinetic energy in Eq.~\eqref{eq:energ_d}, we proceed as follows.

Multiplying the third component of equation of~\eqref{eq:mom_d} by
$w_i$ i.e. the third component of $\tilde{\eta}'(X_i)$ we obtain
$$\Delta x_i \frac{\partial H_i w_i}{\partial t} w_i + \Bigl( F_{H,i+1/2}w_{i+1/2}
- F_{H,i-1/2}w_{i-1/2} \Bigr) w_i + \Delta x_i \left. \nablaswd
  p_{nh}\right|_2w_i = 0.$$
The first term in the preceding equation also writes
\begin{equation}
\frac{\partial H_i w_i}{\partial t} w_i = \frac{\partial }{\partial
  t} \left( \frac{H_i}{2} w_i^2 \right) +
\frac{w_i^2}{2}\frac{\partial H_i }{\partial t}.
\label{eq:energw_1}
\end{equation}
For the fluxes, it comes
\begin{multline}
\Bigl( F_{H,i+1/2}w_{i+1/2}
- F_{H,i-1/2}w_{i-1/2} \Bigr) w_i = F_{H,i+1/2}\frac{w_{i+1/2}^2}{2}
- F_{H,i-1/2} \frac{w_{i-1/2}^2}{2} \\+  F_{H,i+1/2}w_{i+1/2} \left( w_i - \frac{w_{i+1/2}}{2} \right) - F_{H,i-1/2}w_{i-1/2} \left( w_i - \frac{w_{i-1/2}}{2} \right).
\label{eq:energw_2}
\end{multline}
Using the first equation of~\eqref{eq:mom_d} and the definition~\eqref{eq:def_upwind_w}, the sum of Eqs.~\eqref{eq:energw_1}
and~\eqref{eq:energw_2} gives
\begin{multline*}
\Delta x_i \frac{\partial }{\partial
  t} \left( \frac{H_i}{2} w_i^2 \right) + F_{H,i+1/2}\frac{w_{i+1/2}^2}{2}
- F_{H,i-1/2}\frac{w_{i-1/2}^2}{2}  +  \Delta x_i \left. \nablaswd p_{nh}\right|_2w_i =\\
   \frac{1}{2}\left[F_{H,i+1/2}\right]_-
  (w_{i+1} - w_i)^2 - \frac{1}{2}\left[F_{H,i-1/2}\right]_+
  (w_{i} - w_{i-1})^2,
\end{multline*}
with the notations $[a]_+=max(a,0)$, $[a]_-=min(a,0)$ $a=[a]_+ +
[a]_-$. Therefore it comes
\begin{multline}
\Delta x_i \frac{\partial }{\partial
  t} \left( \frac{H_i}{2} w_i^2 \right) + F_{H,i+1/2}\frac{w_{i+1/2}^2}{2}
- F_{H,i-1/2}\frac{w_{i-1/2}^2}{2}  + \Delta x_i \left. \nablaswd
  p_{nh}\right|_2 w_i \leq 0,
\label{eq:energy_w}
\end{multline}
and the left hand side of the preceding equation is exactly the contribution of the vertical kinetic energy over the
energy balance~\eqref{eq:energ_d}.

Adding~\eqref{eq:energ_d} to~\eqref{eq:energy_w} gives
\begin{equation}
\Delta x_i \frac{\partial \tilde{\eta}_{i}}{\partial t} + \left( \tilde{G}_{i+1/2-} -
  \tilde{G}_{i-1/2+}\right) + \begin{pmatrix}
u_i\\
w_i
\end{pmatrix} . \nablaswd p_{nh} \leq 0,
\label{eq:energ_dd}
\end{equation}
with $\tilde{G}_{i+1/2-} = \tilde{G}_{hyd,i+1/2-} + F_{H,i+1/2} \frac{w_{i+1/2}^2}{2}$
and it remains to rewrite the last terms in
Eq.~\eqref{eq:energ_dd}. Using the definitions~\eqref{eq:pnh_sd1},\eqref{eq:pnh_sd2}, we have
\begin{eqnarray}
\Delta x_i \left. \nablaswd p_{nh}\right|_2 w_{i} & = & -\Bigl( \Delta x_{i+1/2}
p_{nh,i+1/2} + \Delta x_{i-1/2} p_{nh,i-1/2} \Bigr) w_{i},
\label{eq:entro_sd_p1}\\
\Delta x_i \left. \nablaswd p_{nh}\right|_1 u_{i} & = & 
H_{i}
  (p_{nh,i+1/2} - p_{nh,i-1/2}) u_{i} \nonumber\\
& & + p_{nh,i+1/2}\bigl(\zeta_{i+1} - \zeta_i\bigr) u_{i} + p_{nh,i-1/2}\bigl(\zeta_{i} - \zeta_{i-1}\bigr) u_{i}\nonumber\\
& = & (Hu) _{i+1/2} p_{nh,i+1/2} - (Hu) _{i-1/2} p_{nh,i-1/2}  \nonumber\\
& & + p_{nh,i+1/2}\bigl(\zeta_{i+1} - \zeta_{i}\bigr) u_{i} + p_{nh,i-1/2}\bigl(\zeta_{i} - \zeta_{i-1}\Bigr) u_{i}\nonumber\\
& & -((Hu)_{i+1} -(Hu)_{i})\frac{p_{nh,i+1/2}}{2} -((Hu)_{i} -(Hu)_{i-1})
\frac{p_{nh,i-1/2}}{2},
\label{eq:entro_sd_p2}
\end{eqnarray}
with
$$(Hu)_{i+1/2} = \frac{(Hu)_{i+1} + (Hu)_i}{2}.$$
The divergence free condition~\eqref{eq:div_d} multiplied by $1/2\, p_{nh,i+1/2}$ gives
\begin{multline}
\frac{p_{nh,i+1/2}}{2}\Bigl( (Hu)_{i+1} - (Hu)_{i}
\Bigr) 
 - \frac{u_i + u_{i+1}}{2}p_{nh,i+1/2} \left( \zeta_{i+1} -
  \zeta_{i}\right) \\
+ \Delta x_{i+1/2} \frac{p_{nh,i+1/2}}{2}  (w_{i+1}+w_i)  = 0.
\label{eq:entro_sd_p3}
\end{multline}
The sum of relations~\eqref{eq:entro_sd_p1}, \eqref{eq:entro_sd_p2}
and~\eqref{eq:entro_sd_p3} gives
\begin{eqnarray*}
\Delta x_i \nablaswd p_{nh}  . \begin{pmatrix}
u_i\\
w_i
\end{pmatrix} & = & \Bigl( (Hu)_{i+1/2} p_{nh,i+1/2} -
   (Hu)_{i-1/2} p_{nh,i-1/2}\Bigr) + d_{i+1/2} - d_{i-1/2},
\end{eqnarray*}
with
\begin{eqnarray*}
d_{i+1/2} & = & \frac{p_{nh,i+1/2}}{2} \Bigl( \Delta x_{i+1/2}
(w_{i+1} -w_i)  - (u_{i+1}-u_i)\bigl(\zeta_{i+1} - \zeta_{i}\bigr)\Bigr),\\
d_{i-1/2} & = & \frac{p_{nh,i-1/2}}{2} \Bigl( \Delta x_{i-1/2} (w_{i}
-w_{i-1}) - (u_i - u_{i-1})\bigl(\zeta_{i} - \zeta_{i-1}\bigr)\Bigl).
\end{eqnarray*}
When the disvergence free condition is satisfied at the boundaries, this proves {\it (iv)}.

Assuming the variables are smooth enough, the quantities $d_{i+1/2}$,$d_{i-1/2}$ satisfy $d_{i+1/2} - d_{i-1/2} = {\cal O}(\Delta x^3)$
and we have
$$\Delta x_i \begin{pmatrix}
u_i\\
w_i
\end{pmatrix} . \nablaswd p_{nh} = \Bigl( (Hu)_{i+1/2}p_{nh,i+1/2} -
  (Hu)_{i-1/2}p_{nh,i-1/2}\Bigr) + {\cal O}(\Delta x)^3,$$
that completes the proof.
\end{proof}

\subsection{The fully discrete scheme}
\label{subsec:fully_discrete}

Now we examine the fully discrete scheme that consists in the coupling
of the semi-discrete schemes described in paragraphs~\ref{subsec:gen_scheme} and~\ref{subsec:semi_d_space}.

\subsubsection{Prediction step}
\label{subsec:prediction}

Using the space discretization defined in paragraph~\ref{subsec:semi_d_space},
we adopt, for the system~\eqref{eq:form_fin1}, the discretization
\begin{equation}
	X^{n+1/2}_i=X_i^n-\sigma_i^n(F^n_{i+1/2-}-F^n_{i-1/2+}),
	\label{eq:upU}
\end{equation}
where $\sigma_i^n=\Delta t^n/\Delta x_i$ is the ratio between the space and
time steps and $F_{i+ 1/2\pm}$ are given by a robust and efficient
discretization of the hyperbolic part with the topography. 

For the discretization of the topography source term in the Saint-Venant
system, several techniques are available. In this paper,  we use the hydrostatic reconstruction (HR scheme for
short)~\cite{bristeau1}, leading to the following expression for the
numerical fluxes
\begin{equation}\begin{array}{l}
	\dsp F^n_{i+1/2-}=\begin{pmatrix} {\cal
          F}_H(X^n_{i+1/2-},X^n_{i+1/2+})\\ {\cal F}_{q_x}(X^n_{i+1/2-},X^n_{i+1/2+})\\ {\cal F}_{H}(X^n_{i+1/2-},X^n_{i+1/2+}) w_{i+1/2}\end{pmatrix}+\begin{pmatrix}0\\
          g\frac{(H_i^n)^2}{2}-\frac{g(H^n_{i+1/2-})^2}{2}\\ 0\end{pmatrix},\\
	\dsp F^n_{i+1/2+}=\begin{pmatrix} {\cal
          F}_H(X^n_{i+1/2-},X^n_{i+1/2+}) \\ {\cal F}_{q_x}(X^n_{i+1/2-},X^n_{i+1/2+})\\ {\cal
          F}_{H}(X^n_{i+1/2-},X^n_{i+1/2+}) w^n_{i+1/2}\end{pmatrix} +\begin{pmatrix}0\\
          g\frac{(H^n_{i+1})^2}{2}-\frac{g(H^n_{i+1/2+})^2}{2}\\ 0\end{pmatrix},
	\label{eq:HRflux}
	\end{array}
\end{equation}
where~\eqref{eq:fluxHw} has been used and ${\cal F}=({\cal F}_H,{\cal F}_{q_x})^T$ is a numerical flux for
the Saint-Venant system without topography. The reconstructed states
\begin{equation}
	X^n_{i+1/2-}=(H^n_{i+1/2-},H^n_{i+1/2-}u^n_i),\qquad
	X^n_{i+1/2+}=(H^n_{i+1/2+},H^n_{i+1/2+}u^n_{i+1}),
	\label{eq:lrstates}
\end{equation}
are defined by
\begin{equation}
	H^n_{i+1/2-}=(H^n_i+z_{b,i}-z_{b,i+1/2})_+,\qquad
	H^n_{i+1/2+}=(H^n_{i+1}+z_{b,i+1}-z_{b,i+1/2})_+,
	\label{eq:hlr}
\end{equation}
and
\begin{equation}
	z_{b,i+1/2}=\max(z_{b,i},z_{b,i+1}).
	\label{eq:zstar}
\end{equation}


\subsubsection{Correction step}
\label{subsec:correction}

For the system~\eqref{eq:form_fin2},\eqref{eq:euler_sw3_d}, we adopt the discretization
\begin{eqnarray}
H^{n+1}_i & = & H^{n+1/2}_i, \label{eq:h_cor}\\
{\bf u}_{i}^{n+1} & = & {\bf u}_{i}^{n+1/2} - \frac{\Delta
  t^n}{H_{i}^{n+1}} \nablaswd p_{nh}^{n+1} ,\label{eq:hu_cor}\\
\divswd \left( {\bf u}^{n+1} \right) & = & 0,\label{eq:div_free_d}
\end{eqnarray}
with $\nablaswd p_{nh}^{n+1}=\nablaswd (p_{nh}^{n+1};H^{n+1/2})$ and $\divswd \left( {\bf u}^{n+1}
\right)=\divswd \left( {\bf u}^{n+1};H^{ n+1/2}\right)$ defined by
Eqs.~\eqref{eq:pnh_sd1}-\eqref{eq:div_ip12}. Then, applying $\divswd$ 
to~\eqref{eq:hu_cor} and using~\eqref{eq:div_free_d} gives the expression for
the elliptic equation under the form
\begin{equation}
\divswd \left( \frac{1}{H^{n+1}}\nabla_{\!\! sw}\, p_{nh}^{n+1}
 \right) = \frac{1}{\Delta t^n} \divswd
\left( {\bf u}^{n+1/2} \right).
\label{eq:elliptic_d}
\end{equation}
The solution of~\eqref{eq:elliptic_d} gives $p_{nh}^{n+1}$
and allows to calculate $(H{\bf u})_{i}^{n+1}$ using~\eqref{eq:hu_cor}.

Omitting the superscript $^{n+1}$, the expression of
$$\Delta_{sw,i+1/2} p_{nh} = \divswd \left(\frac{1}{H}\nabla_{\!\! sw}\, p_{nh}\right),$$
is given by
\begin{equation*}
\begin{split}
-\Delta x_{i+1/2}\Delta_{sw,i+1/2} p_{nh} = & -\frac{H_{i+1}}{\Delta x_{i+1}}\Bigl(
  p_{nh,i+3/2} -  p_{nh,i+1/2} \Bigr) +
  \frac{H_{i}}{\Delta x_{i}}\Bigl( p_{nh,i+1/2} - p_{nh,i-1/2} \Bigr) \\
& - \frac{p_{nh,i+3/2}}{\Delta x_{i+1}} \bigl( \zeta_{i+2}  -\zeta_{i+1}\bigr)
- \frac{p_{nh,i+1/2}}{\Delta x_{i+1}} \bigl( \zeta_{i+1}  - \zeta_{i} \bigr)\\
& + \frac{p_{nh,i+1/2}}{\Delta x_{i}} \bigl( \zeta_{i+1} - \zeta_{i}\bigr)
+ \frac{p_{nh,i-1/2}}{\Delta x_{i}} \bigl( \zeta_{i} - \zeta_{i-1} \bigr)\\
& + \left( \frac{p_{nh,i+3/2} - p_{nh,i+1/2}}{\Delta x_{i+1}} +
\frac{p_{nh,i+1/2} - p_{nh,i-1/2}}{\Delta x_{i}}\right)\bigl(
\zeta_{i+1} - \zeta_{i}\bigr)\nonumber\\
& + \frac{p_{nh,i+3/2}}{H_{i+1}\Delta x_{i+1}} \bigl(\zeta_{i+2} -
  \zeta_{i+1} \bigr) \bigl(\zeta_{i+1} - \zeta_{i} \bigr)\nonumber\\
& + \frac{p_{nh,i+1/2}}{H_{i+1}\Delta x_{i+1}} \bigl(\zeta_{i+1} -
  \zeta_{i}\bigr)^2
+ \frac{p_{nh,i+1/2}}{H_{i}\Delta x_{i}} \bigl(\zeta_{i+1} -
  \zeta_{i}\bigr)^2\nonumber\\
& + \frac{p_{nh,i-1/2}}{H_{i}\Delta x_{i}} \bigl(\zeta_{i+1} -
  \zeta_{i} \bigr) \bigl(\zeta_{i} -
  \zeta_{i-1}\bigr)\nonumber\\
& +\Delta x_{i+1/2}\biggl( \frac{ \Delta x_{i+3/2} p_{nh,i+3/2} +
 \Delta x_{i+1/2} p_{nh,i+1/2}}{\Delta x_{i+1} H_{i+1}} \nonumber\\
& + \frac{
 \Delta x_{i+1/2} p_{nh,i+1/2}  +  \Delta x_{i-1/2}
 p_{nh,i-1/2}}{\Delta x_iH_i}\biggr).
\end{split}
\end{equation*}
And it remains to prove the previous relation is consistent with the
left hand side of Eq.~\eqref{eq:sturm_liouville2}. We rewrite
$\Delta_{sw,i+1/2} p_{nh}$ under the form
\begin{equation*}
\begin{split}
-\Delta x_{i+1/2}\Delta_{sw,i+1/2} p_{nh} = & -\frac{H_{i+1}}{\Delta x_{i+1}}\Bigl(
  p_{nh,i+3/2} -  p_{nh,i+1/2} \Bigr) +
  \frac{H_{i}}{\Delta x_{i}}\Bigl( p_{nh,i+1/2} - p_{nh,i-1/2} \Bigr) \\
& - \frac{p_{nh,i+3/2}}{\Delta x_{i+1}} \Bigl(
\zeta_{i+2} - 2\zeta_{i+1} + \zeta_{i}\Bigr)\\
& - \left( \frac{1}{\Delta x_{i+1}} - \frac{1}{\Delta x_{i}}\right) p_{nh,i+1/2}\Bigl(
\zeta_{i+1} -\zeta_{i}\Bigr) \\
& - \frac{p_{nh,i-1/2}}{\Delta x_{i}} \Bigl(
  \zeta_{i+1} -2\zeta_{i} + \zeta_{i-1}\Bigr)\\
& + \frac{p_{nh,i+3/2}}{H_{i+1}\Delta x_{i+1}} \Bigl(\zeta_{i+2} -
  \zeta_{i+1}\Bigr) \Bigl(\zeta_{i+1} -
  \zeta_{i}\Bigr)\nonumber\\
& + p_{nh,i+1/2} \left( \frac{1}{H_{i+1}\Delta x_{i+1}} + \frac{1}{H_{i}\Delta x_{i}} \right)\Bigl(\zeta_{i+1} -
  \zeta_{i}\Bigr)^2\nonumber\\
& + \frac{p_{nh,i-1/2}}{H_{i}\Delta x_{i}} \Bigl(\zeta_{i+1} -
  \zeta_{i}\Bigr) \Bigl(\zeta_{i} -
  \zeta_{i-1}\Bigr)\nonumber\\
& + \Delta x_{i+1/2} \biggl(
\frac{\Delta x_{i+3/2} p_{nh,i+3/2} +
\Delta x_{i+1/2} p_{nh,i+1/2}}{\Delta x_{i+1}H_{i+1}} \nonumber\\
& + \frac{\Delta x_{i+1/2} p_{nh,i+1/2}  + \Delta x_{i-1/2}
  p_{nh,i-1/2}}{\Delta x_iH_i}\biggr),
\end{split}
\end{equation*}
that is indeed a consistent discretization of the
left hand side of Eq.~\eqref{eq:sturm_liouville2}.

In the case of a regular mesh $\Delta x_i = \Delta x = cst$, the
preceding expression of $\Delta_{sw,i+1/2} p_{nh}$ reduces to
\begin{equation*}
\begin{split}
-\Delta x^2 \Delta_{sw,i+1/2} p_{nh} =& -H_{i+1}\Bigl(
  p_{nh,i+3/2} -  p_{nh,i+1/2} \Bigr) +
  H_{i}\Bigl( p_{nh,i+1/2} - p_{nh,i-1/2} \Bigr) \\
& - p_{nh,i+3/2} \Bigl(
\zeta_{i+2} - 2\zeta_{i+1} + \zeta_{i}\Bigr)\\
& - p_{nh,i-1/2} \Bigl(
  \zeta_{i+1} -2\zeta_{i} + \zeta_{i-1}\Bigr)\\
& + \frac{p_{nh,i+3/2}}{H_{i+1}} \Bigl(\zeta_{i+2} -
  \zeta_{i+1}\Bigr) \Bigl(\zeta_{i+1} -
  \zeta_{i}\Bigr)\nonumber\\
& + p_{nh,i+1/2} \left( \frac{1}{H_{i}} + \frac{1}{H_{i+1}}\right) \Bigl(\zeta_{i+1} -
  \zeta_{i}\Bigr)^2\nonumber\\
& + \frac{p_{nh,i-1/2}}{H_{i}} \Bigl(\zeta_{i+1} -
  \zeta_{i}\Bigr) \Bigl(\zeta_{i} -
  \zeta_{i-1}\Bigr)\nonumber\\
& + \Delta x^2
\left( \frac{p_{nh,i+3/2} +
p_{nh,i+1/2}}{H_{i+1}} + \frac{p_{nh,i+1/2}  + p_{nh,i-1/2}}{H_i}\right).
\end{split}
\end{equation*}

\begin{remark}
The numerical scheme proposed in this paragraph for the correction step is based on a
finite difference strategy and hence cannot be extended to
unstructured meshes in higher dimension. But, based on a variational
formulation of the correction step, the authors have obtained
a finite element version of the
scheme~\eqref{eq:hu_cor}-\eqref{eq:div_free_d} with polynomial
approximations of the
velocities and the pressure satisfying the discrete $inf-sup$
condition. It is not in the scope of this paper to present and
evaluate such a discretization strategy, it is available in a
companion paper~\cite{nora2}.
\end{remark}

\subsubsection{Boundary conditions}
\label{subsec:BC}

It is difficult to define the boundary conditions for the whole
system. Therefore, we first impose boundary conditions for the
hyperbolic part of the system and then we apply suitable boundary
conditions for the elliptic equation governing the non-hydrostatic
pressure $p_{nh}$.

\paragraph{Hyperbolic part}

The definition and the implementation of the boundary conditions used for
the hyperbolic part have been presented in various papers of some
of the authors. The reader can refer to~\cite{coussin}.

\paragraph{Non-hydrostatic part}

For the non-hydrostatic part, the definition of the boundary
conditions means to find boundary conditions for Eq.~\eqref{eq:elliptic_d} and
we adopt the following strategy. Notice that other solutions can be
investigated since the coupling of the boundary conditions between a
hyperbolic step and a parabolic step is far from being obvious.

\noindent\underline{\it Given flux} When the inflow is prescribed~--
for the hyperbolic part~-- we impose for the elliptic equation~\eqref{eq:elliptic_d} a homogeneous Dirichlet type
boundary condition. More precisely, if for $i=0$ or $i=I-1$, $\left. Hu\right|_{i+1/2}^{n+1/2}=Q_0$
is given then we impose $p_{nh,i+1/2}^{n+1}=0$. This choice is imposed by
the relation~\eqref{eq:hu_cor} in order to ensure $(Hu)^{n+1} =
(Hu)^{n+1/2}$ on the neighbouring cell.

\noindent\underline{\it Given water depth} If the water depth is
prescribed for the hyperbolic part i.e. for $i=0$ or $i=I$, $\left. H\right|_{i+1/2}^{n+1/2}=H_0$
is given then we impose for Eq.~\eqref{eq:elliptic_d} a Neumann type boundary condition under the
form $p_{nh,1/2}=p_{nh,3/2}=0$ or $p_{nh,I-1/2}=p_{nh,I+1/2}=0$.



\subsection{The discrete $inf-sup$ condition}

Using a matrix notation for the shallow water gradient operator
$$\nablasw p_{nh}^{n+1} = B^T p_{nh}^{n+1},$$
we get
$$\divsw {\bf u}^{n+1}  = B {\bf u}^{n+1},$$
where suitable boundary conditions are assumed. Therefore, defining
$$\Lambda = \mbox{diag} (H_i^{-1}),$$
 the fully
discrete scheme obtained from~\eqref{eq:upU},\eqref{eq:h_cor}-\eqref{eq:div_free_d} can
be rewritten under the form
\begin{equation}
\begin{pmatrix}
\frac{1}{\Delta t} & 0 & 0\\
0 & \frac{1}{\Delta t} & B^T\\
0 & B \Lambda & 0
\end{pmatrix}
\begin{pmatrix}
H^{n+1}\\
(H{\bf u})^{n+1}\\
p_{nh}^{n+1}
\end{pmatrix} = \begin{pmatrix}
\frac{H^n}{\Delta t} + D_H(X^n)\\
\frac{(H{\bf u})^n}{\Delta t} + D_{H\underline{u}}(X^n)\\
0
\end{pmatrix},
\label{eq:inf_sup}
\end{equation}
where $D_H$,$D_{H\underline{u}}$ refer to the numerical discretization
of the hyperbolic part. The fact that the matrix $B\Lambda B^T$ is invertible is related to
the inf-sup condition~\cite{lbb}, see~\cite{nora2} where this property is
investigated in details.

\begin{remark}
Instead of
the scheme~\eqref{eq:inf_sup},
a fully implicit version~-- including the hyperbolic part~-- may be considered. But such a discretization would imply to have an
implicit treatment of the hyperbolic part of the proposed model
corresponding to the Saint-Venant system. And an efficient and robust
implicit solver for the Saint-Venant system is hardly accessible.
\end{remark}

\subsection{Stability of the scheme}

For the numerical scheme detailed in paragraphs~\ref{subsec:prediction}
and~\ref{subsec:correction} we have the following
proprosition.

\begin{proposition}
The scheme~\eqref{eq:upU},\eqref{eq:h_cor}-\eqref{eq:div_free_d}
\begin{itemize}
\item[{\it (i)}] preserves the nonnegativity of the water depth
  $H_i^{n+1}\geq 0$, $\forall i$, $\forall n$,
\item[{\it (ii)}] preserves the steady state of the lake at rest,
\item[{\it (iii)}] is consistent with the
  model~\eqref{eq:euler_11}-\eqref{eq:wbar1}.
\end{itemize}
\label{prop:scheme_prop}
\end{proposition}

\begin{proof}[Proof of prop.~\ref{prop:scheme_prop}]
{\it (i)} The statement that ${\cal F}$ preserves the nonnegativity of
the water depth means exactly that
$${\cal F}_H(H_i = 0, u_i, H_{i+1}, u_{i+1}) - {\cal F}_H(H_{i-1},
u_{i-1}, H_i = 0, u_i) \leq 0,$$
for all choices of the other arguments. From~\eqref{eq:upU},\eqref{eq:HRflux}, we need to check that
$${\cal F}_H(X_{i+1/2−}^n,X_{i+1/2+}^n) - {\cal
  F}_H(X_{i-1/2−}^n,X_{i-1/2+}^n) \leq 0,$$
whenever $H_i^n=0$. And this property holds since
from~\eqref{eq:hlr},\eqref{eq:zstar}
$H_i=0$ implies $H_{i+1/2-}=H_{i-1/2+}=0$.

{\it (ii)} When $u_i^n=0$ for all $i$, the properties of the hydrostatic reconstruction technique ensures
$$F^n_{i+1/2-} = F^n_{i-1/2+},$$
in~\eqref{eq:upU} and hence $X_i^{n+1/2} = X_i^{n}$ moreover the
scheme~\eqref{eq:h_cor},\eqref{eq:hu_cor},\eqref{eq:elliptic_d} gives
$$X_i^{n+1} = X_i^{n+1/2},$$
proving the scheme is well-balanced.

{\it (iii)} The numerical flux ${\cal F}$ being consistent with the
homogeneous Saint-Venant system, the hydrostatic reconstruction
associated with ${\cal F}$ gives a consistent discretization of the
Saint-Venant system with the topography source term. The
discretizations~\eqref{eq:hu_cor},\eqref{eq:div_free_d} being
obviously consistent with the remaining part, this
proves the result.
\end{proof}

\subsection{$H\rightarrow 0$}
\label{subsec:H0}

When $H$ tends to 0, the correction step~\eqref{eq:hu_cor} is no longer
valid and we propose a modified version
of~\eqref{eq:hu_cor},\eqref{eq:div_free_d} under the form
\begin{eqnarray*}
{\bf u}_{i}^{n+1} & = & {\bf u}_{i}^{n+1/2} - \Delta t^n\frac{1}{H_{i}^{n+1}} \nablaswdepsi p_{nh}^{n+1} ,\label{eq:hu_cor_mod}\\
\divswd \left( {\bf u}^{n+1}\right) & = & 0, \label{eq:div_free_d_mod}
\end{eqnarray*}
and
\begin{eqnarray*}
\frac{\Delta x_i}{H_i^{n+1}}\left. \nablaswdepsi
    p_{nh}^{n+1} \right|_1 & = & 
    p_{nh,i+1/2}^{n+1} -
    p_{nh,i-1/2}^{n+1}\\
& & + \frac{\1_{H_i^{n+1} \geq \varepsilon}}{H_{i,\varepsilon}^{n+1}}
\biggl( p^{n+1}_{nh,i+1/2}
\Bigl(\zeta^{n+1}_{i+1} - \zeta^{n+1}_{i}  \Bigr)+ p^{n+1}_{nh,i-1/2}
\Bigl(\zeta^{n+1}_{i} - \zeta^{n+1}_{i-1} \Bigr)\biggr),\\
\frac{\Delta x_i}{H_i^{n+1}} \left. \nablaswdepsi p_{nh}^{n+1} \right|_2 & = & -\frac{1}{H_{i,\varepsilon}^{n+1}}\left(\Delta x_{i+1/2}
p^{n+1}_{nh,i+1/2} + \Delta x_{i-1/2} p^{n+1}_{nh,i-1/2}\right) ,
\end{eqnarray*}
with $\varepsilon$ being a constant $\varepsilon=cst >0$ and
$H_\varepsilon = \max(H,\varepsilon)$.

In order to ensure the total pressure
$$\frac{g}{2}H + \overline{p}_{nh},$$
remains non negative, we add the constraint
$$p_{nh,i+1/2}^{n+1}=0 \quad{\rm when}\quad
\frac{g}{2}\min(H_i^{n+1},H_{i+1}^{n+1}) + p_{nh,i+1/2}^{n+1} \leq 0,$$
to the solution of the elliptic equation~\eqref{eq:elliptic_d}.
Notice that in all the numerical tests presented in this paper this
constraint is not active meaning the total pressure remains non negative.

\section{Fully discrete entropy inequality}
\label{sec:entropy}

We have precised in~paragraph~\ref{subsec:fully_discrete} a general scheme for the
resolution of the non-hydrostatic model. In this paragraph we study
the properties of the proposed scheme in the context of one particular
solver for the hyperbolic part, namely the kinetic solver, since it allows to ensure
stability properties among which are entropy inequalities
(semi-discrete and fully discrete)~\cite{JSM_entro}.

Looking for a kinetic interpretation of the HR scheme,
we would like to write down a kinetic scheme for Eq.~\eqref{eq:gibbs}
such that the associated macroscopic scheme is exactly \eqref{eq:upU}-\eqref{eq:HRflux} with
the definitions~\eqref{eq:lrstates}-\eqref{eq:zstar}.

We drop the superscript $^n$ and keep superscripts $n+1$ and $n+1/2$. We denote $M_i=M(H_i,u_i,\xi)$, $M_{i+1/2-}=M(H_{i+1/2-},
u_i,\xi)$, $M_{i+1/2+}=M(H_{i+1/2+}, u_{i+1},\xi)$, $f_i^{n+1/2-}=f_i^{n+1/2-}(\xi)$,
and we consider the scheme
\begin{equation}\begin{array}{l}
	\dsp f_i^{n+1/2-}=M_i-\sigma_i\biggl(\xi\1_{\xi<0}M_{i+1/2+}+\xi\1_{\xi>0}M_{i+1/2-}+\delta M_{i+1/2-}\\
	\dsp\mkern 160mu -\xi\1_{\xi>0}M_{i-1/2-}-\xi\1_{\xi<0}M_{i-1/2+}-\delta M_{i-1/2+}\biggr).
	\label{eq:kinups}
	\end{array}
\end{equation}
In this formula, $\delta M_{i+1/2\pm}$ are defined by
$$\delta M_{i+1/2-} = (\xi - u_i) (M_i - M_{i+1/2-}),\qquad \delta
M_{i+1/2+} = (\xi - u_{i+1}) (M_{i+1} - M_{i+1/2+}),$$ 
and are assumed to satisfy the moment relations
\begin{equation}
	\int_\R\delta M_{i+1/2-}\,d\xi=0,\quad
	\int_\R \xi\,\delta M_{i+1/2-}\,d\xi=g\frac{H_i^2}{2}-g\frac{H_{i+1/2-}^2}{2},
	\label{eq:intdeltaM-}
\end{equation}
\begin{equation}
	\int_\R\delta M_{i-1/2+}\,d\xi=0,\quad
	\int_\R\xi\,\delta M_{i-1/2+}\,d\xi=g\frac{H_i^2}{2}-g\frac{H_{i-1/2+}^2}{2}.
	\label{eq:intdeltaM+}
\end{equation}
Defining the update as
\begin{equation}
	\begin{pmatrix} H\\Hu\end{pmatrix}^{n+1/2}_i=\frac{1}{\Delta x_i}\int_{x_{i-1/2}}^{x_{i+1/2}}\int_\R \kxi f(t^{n+1/2-},x,\xi)\,dxd\xi,
	\label{eq:updatekin}
\end{equation}
and
\begin{equation}
	f^{n+1/2-}_i(\xi)=\frac{1}{\Delta x_i}\int_{x_{i-1/2}}^{x_{i+1/2}}f(t^{n+1/2-},x,\xi)\,dx,
	\label{eq:f_in+1-}
\end{equation} 
the formula \eqref{eq:updatekin} can then be written
\begin{equation}
	\begin{pmatrix} H\\Hu\end{pmatrix}^{n+1/2}_i=\int_\R \kxi f^{n+1/2-}_i(\xi)\,d\xi,
	\label{eq:Un+1-i}
\end{equation}
and using~\eqref{eq:fluxHw} we also define
\begin{equation}
(H_i w_i)^{n+1/2} = H_i w_i - \sigma_i \Bigl( w_{i+1/2}F^{kin}_{H,i+1/2} 
- w_{i-1/2}F^{kin}_{H,i-1/2} \Bigr),
\label{eq:hwn+1-i}
\end{equation}
with
$$F^{kin}_{H,i+1/2} = \int_{\R} \xi \bigl(
  \1_{\xi>0}M_{i+1/2-}+\1_{\xi<0}M_{i+1/2+} \bigr) d\xi.$$
Finally, relations~\eqref{eq:kinups}-\eqref{eq:hwn+1-i} give an explicit formula for the
prediction step~\eqref{eq:upU} and we have the following proposition
that is proved in~\cite[Corollary~3.8]{JSM_entro} (the two constants
$v_m$ and $C_\beta$ are precised therein).

\begin{proposition}
Let $v_m$ and $C_\beta$ be two constants. For $\sigma_i$ small enough, the numerical scheme~\eqref{eq:upU} based on the kinetic
description~\eqref{eq:kinups}-\eqref{eq:hwn+1-i} and the HR
technique~\eqref{eq:hlr},\eqref{eq:zstar} satisfies the fully discrete
entropy inequality
\begin{equation}\begin{array}{l}
	\dsp \tilde{\eta}(X_i^{n+1/2}) \leq \tilde{\eta}(X_i)
	-\sigma_i\Bigl(\widetilde G_{i+1/2}-\widetilde G_{i-1/2}\Bigr)\\
	\dsp\hphantom{\eta(U_i^{n+1})+gz_{b,i}H_i^{n+1}\leq}
	+C_\beta (\sigma_iv_m)^2\biggl(g(z_{b,i+1}-z_{b,i})^2+g(z_{b,i}-z_{b,i-1})^2 \biggr),
	\label{eq:estentrint}
	\end{array}
\end{equation}
with $\widetilde G_{i+1/2} = \widetilde G_{hyd,i+1/2} + F_{H,i+1/2} \frac{w_{i+1/2}^2}{2}$ and
\begin{equation}
	\widetilde G_{hyd,i+1/2}=\int_{\xi<0}\xi H(M_{i+1/2+},z_{b,i+1/2})\,d\xi+\int_{\xi>0}\xi H(M_{i+1/2-},z_{b,i+1/2})\,d\xi.
	\label{eq:kingtilde}
\end{equation}
\label{prop:entropy_predic}
\end{proposition}

\begin{remark}
Let us notice that the quadratic error term has the following key properties: it vanishes identically when $z=cst$ (no topography)
or when $\sigma_i\rightarrow 0$ (semi-discrete limit), and as soon as the topography
is Lipschitz continuous, it tends to zero strongly when the grid size tends to $0$
(consistency with the continuous entropy inequality~\eqref{eq:euler_55}),
even for non smooth solutions.
\end{remark}

For the correction step, we have the following discrete energy
balance.

\begin{proposition}
The numerical scheme~\eqref{eq:hu_cor},\eqref{eq:div_free_d} satisfies
the following inequality
\begin{multline*}
	\tilde{\eta}(X_i^{n+1}) \leq \tilde{\eta}(X^{n+1/2}_i)\\
	-\sigma_i \Bigl((Hu)_{i+1/2}^{n+1}p_{nh,i+1/2}^{n+1} -
  (Hu)^{n+1}_{i+1/2}p_{nh,i-1/2}^{n+1}\Bigr) - C_1^2 \bigl((\Delta t^n)^2 + C_2
  (\Delta x_i)^3\bigr).
	\label{eq:entro_correc}
\end{multline*}
Moreover, for $\sigma_i\approx 1$, $\Delta t^n$ small enough and assuming suitable boundary
conditions, we have
$$\sum_i \Delta x_i \left( \tilde{\eta}(X_i^{n+1}) -
  \tilde{\eta}(X^{n+1/2}_i) \right) \leq 0.$$
\label{prop:entropy_correc}
\end{proposition}

\begin{corollary}
The numerical scheme detailed in paragraphs~\ref{subsec:prediction} and~\ref{subsec:correction}
satisfies the fully discrete entropy inequality
\begin{multline*}
\tilde{\eta}(X_i^{n+1}) \leq \tilde{\eta}(X_i)
	-\sigma_i\Bigl(\widehat G_{i+1/2}-\widehat G_{i-1/2}\Bigr)\\
	\dsp\hphantom{\eta(U_i^{n+1})+gz_{b,i}H_i^{n+1}\leq}
	+C_\beta (\sigma_iv_m)^2\biggl(g(z_{b,i+1}-z_{b,i})^2+g(z_{b,i}-z_{b,i-1})^2 \biggr) - C_1^2 (\sigma_i^2 + C_2
  \Delta t^n),
\end{multline*}
with $\widehat G_{i+1/2} = \widetilde G_{i+1/2} + (Hu)_{i+1/2}p_{nh,i+1/2}$.
\label{corol:entropy_tot}
\end{corollary}

\begin{proof}[Proof of prop.~\ref{prop:entropy_predic}]
With $\eta_{hyd}(X)$ defined by~\eqref{eq:eta_hyd}, we start from the inequality
\begin{equation}\begin{array}{l}
	\dsp \tilde{\eta}_{hyd}(X_i^{n+1/2})\leq \tilde{\eta}_{hyd}(X_i)
	-\sigma_i\Bigl(\widetilde G_{hyd,i+1/2}-\widetilde G_{hyd,i-1/2}\Bigr)\\
	\dsp\hphantom{\eta(U_i^{n+1})+gz_{b,i}H_i^{n+1}\leq}
	+C_\beta (\sigma_iv_m)^2\biggl(g(z_{b,i+1}-z_{b,i})^2+g(z_{b,i}-z_{b,i-1})^2 \biggr),
	\label{eq:estentrint_hyd}
	\end{array}
\end{equation}
that is proved in~\cite[Corollary~3.8]{JSM_entro}. Equation~\eqref{eq:estentrint_hyd} corresponds to a fully discrete entropy inequality for the
Saint-Venant system including the topography source term.

Multiplying~\eqref{eq:hwn+1-i} by $w_i$ leads to
$$
(H_i w_i)^{n+1/2}w_i = H_i w_i^2 - \sigma_i \Bigl( w_{i+1/2}w_i F^{kin}_{H,i+1/2}
- w_{i-1/2}w_iF^{kin}_{H,i-1/2},\Bigr),
$$
with
\begin{multline*}
(H_i w_i)^{n+1/2}w_i - H_i w_i^2 = \frac{H_i^{n+1/2}}{2}
  \bigl(w_i^{n+1/2}\bigr)^2 -
\frac{H_i}{2}w_i^2 \\
+ \frac{w_i^2}{2} (H_i^{n+1/2} - H_i)
- \frac{H_i^{n+1/2}}{2} (w_i^{n+1/2} - w_i)^2,
\end{multline*}
and
\begin{multline*}
w_{i+1/2}w_i F^{kin}_{H,i+1/2} - w_{i-1/2}w_i F^{kin}_{H,i-1/2}  =
\frac{w_{i+1/2}^2}{2}F^{kin}_{H,i+1/2}  - 
\frac{w_{i-1/2}^2}{2}F^{kin}_{H,i-1/2} \\
+ w_{i+1/2} F^{kin}_{H,i+1/2} \left( w_i - \frac{w_{i+1/2}}{2}\right)
- w_{i-1/2} F^{kin}_{H,i-1/2} \left( w_i - \frac{w_{i-1/2}}{2}\right).
\end{multline*}
The sum of the two previous relations gives
\begin{multline}
\frac{H_i^{n+1/2}}{2}\bigl(w_i^2\bigr)^{n+1/2} -
\frac{H_i}{2}w_i^2  + \sigma_i \Bigl(\frac{w_{i+1/2}^2}{2}F^{kin}_{H,i+1/2}  - 
\frac{w_{i-1/2}^2}{2}F^{kin}_{H,i-1/2}\Bigr)\\
\leq \frac{\sigma_i}{2}\left[F^{kin}_{H,i+1/2}\right]_-
  (w_{i+1} - w_i)^2 - \frac{\sigma_i}{2}\left[F^{kin}_{H,i-1/2}\right]_+
  (w_{i} - w_{i-1})^2 + \frac{H_i^{n+1/2}}{2} (w_i^{n+1/2} - w_i)^2.\label{eq:estim_entro_w}
\end{multline}
It remains to estimate, when $H_i^{n+1/2}>0$, the quantity
$$\frac{H_i^{n+1/2}}{2} (w_i^{n+1/2} - w_i)^2,$$
in the r.h.s. of Eq.~\eqref{eq:estim_entro_w}.

We have
\begin{eqnarray*}
w_i^{n+1/2} - w_i  & = & \frac{1}{H_i^{n+1/2}} \Bigl( (Hw)_i^{n+1/2} -
(Hw)_i - (H_i^{n+1/2} -H_i) w_i\Bigr)\\
& = & \frac{\sigma_i}{H_i^{n+1/2}} \Bigl( F^{kin}_{H,i+1/2} (w_i - w_{i+1/2}) -
F^{kin}_{H,i-1/2} (w_i - w_{i-1/2})\Bigr),
\end{eqnarray*}
and hence
\begin{eqnarray*}
\frac{H_i^{n+1/2}}{2} (w_i^{n+1/2} - w_i)^2  & \leq & \frac{\sigma_i^2}{H_i^{n+1/2}} \Bigl( (F^{kin}_{H,i+1/2})^2 (w_i - w_{i+1/2})^2 +
(F^{kin}_{H,i-1/2})^2 (w_i - w_{i-1/2})^2\Bigr).
\end{eqnarray*}

Therefore, for $\sigma_i$ small enough, the r.h.s. of
Eq.~\eqref{eq:estim_entro_w} is non
positive with
\begin{equation*}
\frac{H_i^{n+1/2}}{2}\bigl( w_i^2\bigr)^{n+1/2} -
\frac{H_i}{2}w_i^2  + \sigma_i \Bigl(\frac{w_{i+1/2}^2}{2}F^{kin}_{H,i+1/2}  - 
\frac{w_{i-1/2}^2}{2}F^{kin}_{H,i-1/2}\Bigr)
\leq -C^2_1\sigma_i(1-C_2^2\sigma_i).
\end{equation*}
The previous relation coupled with~\eqref{eq:estentrint_hyd} gives the result.
\end{proof}

\begin{proof}[Proof of prop.~\ref{prop:entropy_correc}]
We start from relation~\eqref{eq:hu_cor} multiplied by $H_i^{n+1}{\bf
  u}_i^{n+1}$, this gives
\begin{multline*}
\left(\frac{H_i}{2}{\bf u}_i^2\right)^{n+1} - \left(\frac{H_i}{2}{\bf
    u}_i^2\right)^{n+1/2} +\Delta t^n \nablaswd p_{nh}^{n+1} . {\bf
  u}^{n+1}_i = - \frac{H_i^{n+1/2}}{2}\left({\bf
    u}_i^{n+1}  - {\bf u}_i^{n+1/2} \right)^2,
\end{multline*}
with the notation ${\bf u}^2 = {\bf u}.{\bf u}$.

Omitting in this part the
superscript $^{n+1}$ and as in the the proof
of prop.~\ref{prop:entropy_semi_d}, simple manipulations give
\begin{eqnarray*}
\Delta x_i \nablaswd p_{nh}  . \begin{pmatrix}
u_i\\
w_i
\end{pmatrix} & = & \Bigl(  (Hu)_{i+1/2}p_{nh,i+1/2} -
  (Hu)_{i-1/2}p_{nh,i-1/2}\Bigr) + d_{i+1/2} - d_{i-1/2},
\end{eqnarray*}
with
\begin{eqnarray*}
d_{i+1/2} & = & \frac{p_{nh,i+1/2}}{2} \Bigl( \Delta x_{i+1/2}
(w_{i+1} -w_i)  - \frac{u_{i+1}-u_i}{2}\bigl(H_{i+1} +
  2z_{b,i+1} - (H_{i}+2z_{b,i})\bigr)\Bigr),\\
d_{i-1/2} & = & \frac{p_{nh,i-1/2}}{2} \Bigl( \Delta x_{i-1/2} (w_{i}
-w_{i-1}) - \frac{u_i - u_{i-1}}{2}\bigl(H_{i} +
  2z_{b,i} - (H_{i-1}+2z_{b,i-1})\bigr)\Bigl),
\end{eqnarray*}
proving the result.

Assuming the variables are smooth enough, the quantities $d_{i+1/2}$,$d_{i-1/2}$ satisfy $d_{i+1/2} - d_{i-1/2} = {\cal O}(\Delta x^3)$
and we have
$$\Delta x_i \begin{pmatrix}
u_i\\
w_i
\end{pmatrix} . \nablaswd p_{nh} = \Bigl( (Hu)_{i+1/2}p_{nh,i+1/2} -
   (Hu)_{i-1/2}p_{nh,i-1/2}\Bigr) + {\cal O}(\Delta x)^3,$$
that completes the proof.

Notice that in the limit $\Delta t^n\rightarrow 0$, $\sigma_i \approx
1$, the inequality
$$ -\frac{H_i^{n+1/2}}{2}\left({\bf
    u}_i^{n+1}  - {\bf u}_i^{n+1/2} \right)^2 + \sigma_i(d_{i+1/2} -
d_{i-1/2}) \leq 0,$$
holds, meaning the correction step ensures a decrease of the entropy.
\end{proof}

\begin{proof}[Proof of corollary~\ref{corol:entropy_tot}]
The sum of the two inequalities obtained in props.~\ref{prop:entropy_predic} and~\ref{prop:entropy_correc}
gives the result.
\end{proof}

\section{Analytical solutions}
\label{sec:anal}

Stationary and time dependent analytical solutions are available for
the model~\eqref{eq:euler_sw1}-\eqref{eq:euler_sw3}, see~\cite{JSM_nhyd} and references
therein. In this section we only briefly recall some of them, they will be
very useful to evaluate the properties of the proposed numerical
scheme, see paragraph~\ref{sec:num_valid}.

\subsection{Time dependent analytical solution}

\subsubsection{Parabolic bowl}
\label{subsec:thacker_nh}

The functions defined by
\begin{eqnarray}
H(x, t) & = & \max \left( H_0 -\frac{b_2}{2} \left( x-\int^t_{\tilde{t}^0}
    f(t_1)dt_1\right)^2, 0 \right), \label{eq:sol1}\\
\overline{u}(x,t) & = & f(t)\1_{H>0}, \label{eq:sol2}\\
\overline{w}(x,t) & = & b_2 x f(t) \1_{H>0}, \label{eq:sol3}\\
z_b(x) & = & b_1+\frac{b_2}{2} x^2, \label{eq:sol4}\\
\overline{p}_{nh}(x,t) & = & \frac{b_2f^2}{2}H^2, \label{eq:sol5}\\
s(x,z,t) & = & b_2 x \frac{d f}{ d t},  \label{eq:sol6}
\end{eqnarray}
where $H_0>0,b_1,b_2$ are constants and the function $f$ satisfies the ODE
\begin{equation}
\frac{d f}{dt} +b_2 (g+b_2 f^2) \int^t_{\tilde{t}^0}
f(t_1)dt_1= 0,\quad f(t_0) = f^0, \;\; \tilde{t}^0\in \R,
\label{eq:f}
\end{equation}
are solutions of the system
\begin{align}
&\frac{\partial H}{\partial t} + \frac{\partial}{\partial x} \bigl(H\overline{u}\bigr) = 0, \label{eq:thacker_11}\\
&\frac{\partial}{\partial t}(H\overline{u}) + \frac{\partial}{\partial x}\left(H \overline{u}^2
+ \frac{g}{2}H^2 + H\overline{p}_{nh}\right) = - (gH +
2 \overline{p}_{nh})\frac{\partial
  z_b}{\partial x},\label{eq:thacker_22}\\
&\frac{\partial}{\partial t}(H\overline{w})  + \frac{\partial}{\partial
  x}(H\overline{w}\overline{u}) = 2
\overline{p}_{nh} + Hs,\label{eq:thacker_33}\\
& \frac{\partial (H\overline{u})}{\partial x} -
\overline{u} \frac{\partial (H+2z_b)}{\partial x} + 2\overline{w} = 0,
\label{eq:thacker_44}
\end{align}
that corresponds to~\eqref{eq:euler_sw1}-\eqref{eq:euler_sw2}
completed with a source term $Hs$ in the momentum equation~\eqref{eq:thacker_33}.

\subsubsection{Solitary wave solutions}
\label{subsec:solitary}

The system~\eqref{eq:euler_sw1}-\eqref{eq:euler_sw2} admits
solitary waves having the form
\begin{eqnarray}
H =  H_0 + a \left( \sech\left(\frac{x-c_0t}{l}\right) \right)^2,\label{eq:soliton1}\\
\overline{u}  =  c_0 \left( 1 - \frac{d}{H} \right),\label{eq:soliton2}\\
\overline{w} = -\frac{ac_0 d}{l H}\sech \left(\frac{x-c_0t}{l}\right)
\sech' \left(\frac{x-c_0t}{l}\right),\label{eq:soliton3}\\
\overline{p}_{nh} = \frac{ac_0^2d^2}{2l^2H^2} \left( (2H_0-H)
  \left(\sech' \left(\frac{x-c_0t}{l}\right)\right)^2\right. \nonumber\\
\qquad\left. + H \sech \left(\frac{x-c_0t}{l}\right) \sech'' \left(\frac{x-c_0t}{l}\right)\right),\label{eq:soliton4}
\end{eqnarray}
where $\varphi'$ denotes the derivative of function $\varphi$,
$$c_0 = \frac{l}{d}\sqrt{\frac{g H^3_0}{l^2-H^2_0}},\quad\quad
a=\frac{H^3_0}{l^2-H^2_0},$$
and $(d,l,H_0) \in \R^3$ are given constants with $l>H_0>0$.

\section{Numerical simulations}
\label{sec:num_valid}

A complete validation of the proposed numerical technique is not in
the scope of this paper and will be investigated in a forthcoming
paper. We focus on two typical situations, described in
paragraphs~\ref{subsec:thacker_nh} and~\ref{subsec:solitary} where
analytical solutions exist.

For the numerical test, we use a 2$^{nd}$ order extension of the space
discretization for the prediction step. The second order extension is
built as in~\cite{bristeau}. For the second order extension of the
time scheme, we use a Heun type scheme, see~\cite{heun}.

\subsection{The parabolic bowl}

At the discrete level, the analytical solution given in paragraph~\ref{subsec:thacker_nh}
(see also~\cite{JSM_nhyd}) is particularly difficult to
capture. Indeed, it is a non stationary solution and the flow exhibits wet/dry interfaces all along the simulation.

With the parameter values $H_0=1$, $a=1$, $b_1=0$, $b_2 =1$, $t^0 = 0$ over the geometrical
domain $[-2,2]$ and with the initial conditions (see Fig.~\ref{fig:thacker_init})
\begin{eqnarray*}
&& \int_{\tilde{t}^0}^{t^0} f(t)\ dt = \frac{a}{\sqrt{g b_2}},\\
&& f(t_0) = f^0 = 0,
\end{eqnarray*}
we have calculated~-- with a simple Runge-Kutta scheme~-- the
solution of the ODE~\eqref{eq:f}. The solution has been calculated with a
very fine time discretization and thus can be considered as a
reference solution, very close to the analytical solution of
Eq.~\eqref{eq:f}. This means we have at our disposal an
analytical solution for the system~\eqref{eq:thacker_11}-\eqref{eq:thacker_44}.
\begin{figure}[hbtp]
\begin{center}
\includegraphics[width=12cm]{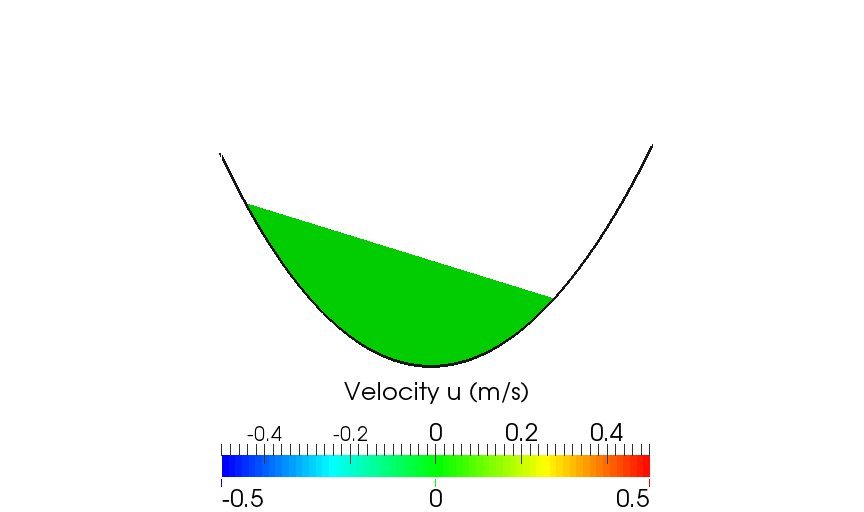}
\end{center}
\caption{Initial conditions for the simulation of the ``parabolic bowl''
  (parabolic bottom, water depth and null horizontal velocity).}
\label{fig:thacker_init}
\end{figure}

To illustrate the behavior of the solution in such a situation, we give over
Fig~\ref{fig:conv_thacker_0_8} the variations along time of the water
depth at $x_0 = 0.8$ m for a mesh of 80 cells which is a rather
coarse mesh.

In order to evaluate the convergence rate of the simulated solution
towards the analytical one, we plot the error rate versus the space
discretization. We have plotted in Fig.~\ref{fig:conv_thacker_nh} the
$\mbox{log}(L^1-error)$ over the water depth at time $T=10$ seconds~-- corresponding to more than 5 periods~-- versus
$\mbox{log}(h_0/h_i)$ for the first and second-order schemes and they
are compared to the theoretical order. We denote by $h_i$ the average
cell length, $h_0$ is the average cell length of the
coarser space discretization. These errors have been computed on 6 meshes with $20$,
$40$, $80$, $120$, $300$ and $400$ cells. 

For this test case, the errors due to the time and space schemes are
combined so the convergence rate of the simulated solution towards the
analytical one is more difficult to analyze. Despite this fact, it appears that the
computed convergence rates are close to the theoretical ones. It
emphasizes the performance of the proposed numerical technique.

 Over Fig~\ref{fig:conv_thacker_nh_x0}, we have plotted the variations along time of the
$\mbox{log}(L^1-error)$ over the water depth for the mesh with 80
cells calculated at node $x_0 = 0.8$ m i.e. the quantity
$$t \mapsto \log \left( 100 \frac{\left| H_{sim}(x_0,t) -
      H_{anal}(x_0,t) \right|}{H_{sim}(x_0,t)}\right).$$
It appears that this quantity does not increase with time and remains bounded.

\begin{figure}[hbtp]
\begin{center}
\includegraphics[width=10.5cm]{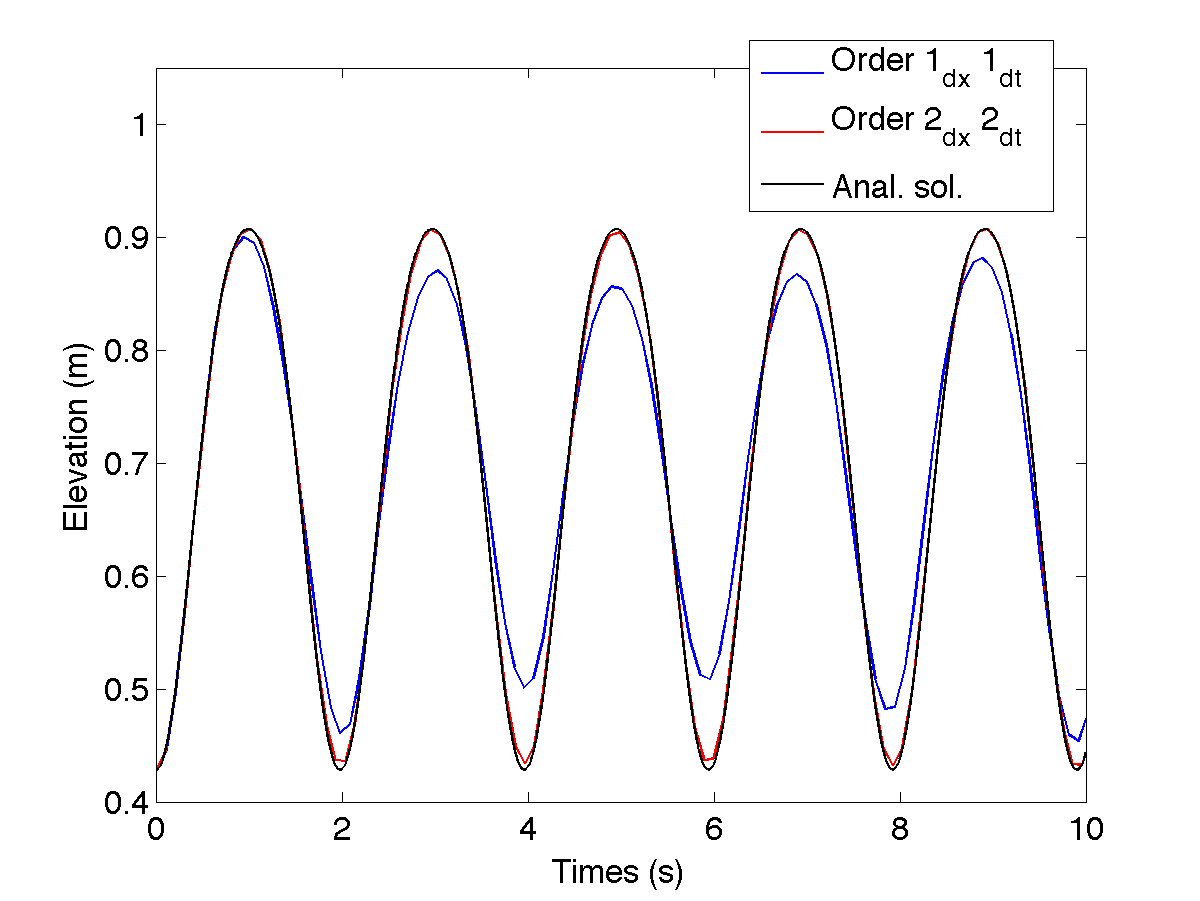}
\end{center}
\caption{Parabolic bowl: variations of $t \mapsto H(0.8,t)$ -
  analytical solution and simulated one with the
  first order (space and time) and second order (space and time) schemes.}
\label{fig:conv_thacker_0_8}
\end{figure}

\begin{figure}[hbtp]
\begin{center}
\includegraphics[width=10.5cm]{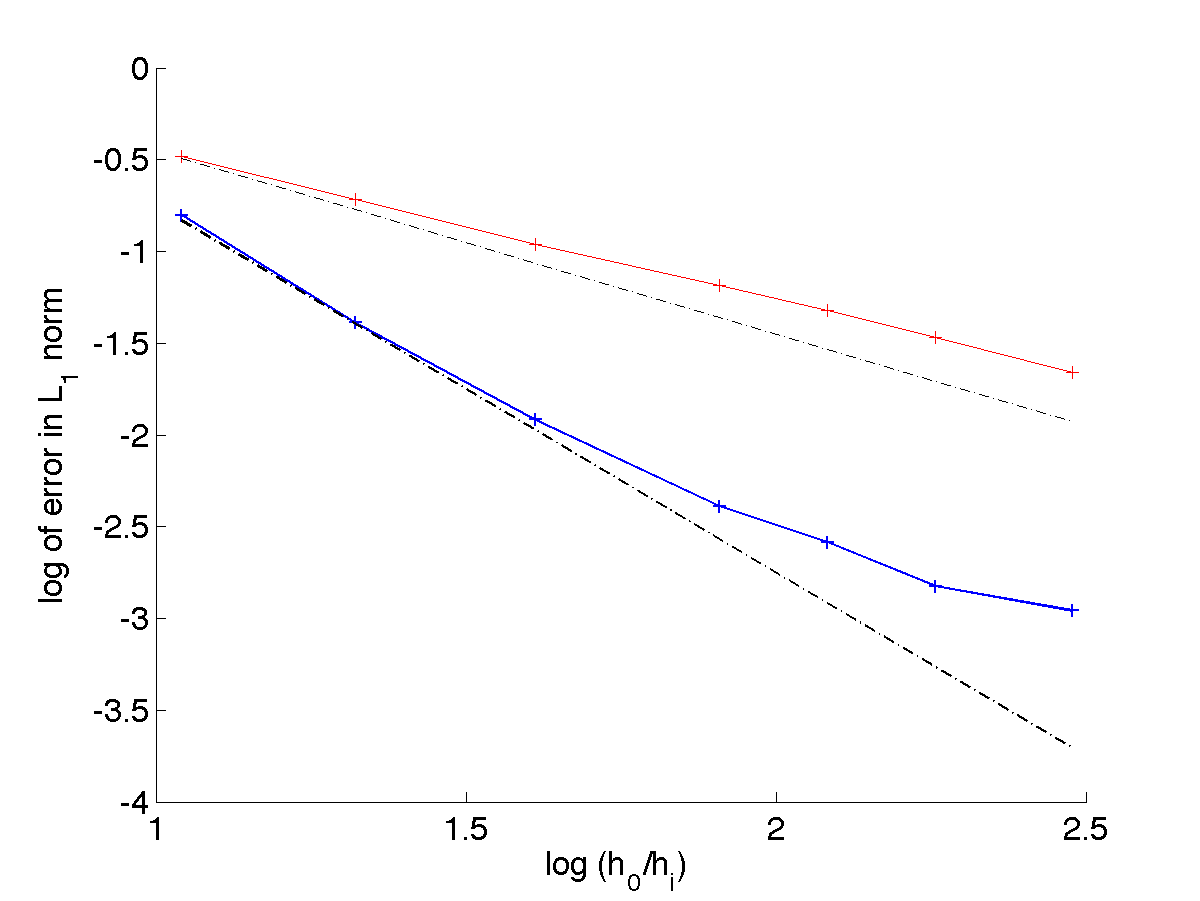}
\end{center}
\caption{Parabolic bowl: convergence rates to the reference solution, 1$^{st}$ order
schemes (space and time) and 2$^{nd}$ order
schemes (space and time), '-.' theoretical order.}
\label{fig:conv_thacker_nh}
\end{figure}

\begin{figure}[hbtp]
\begin{center}
\includegraphics[width=10.5cm]{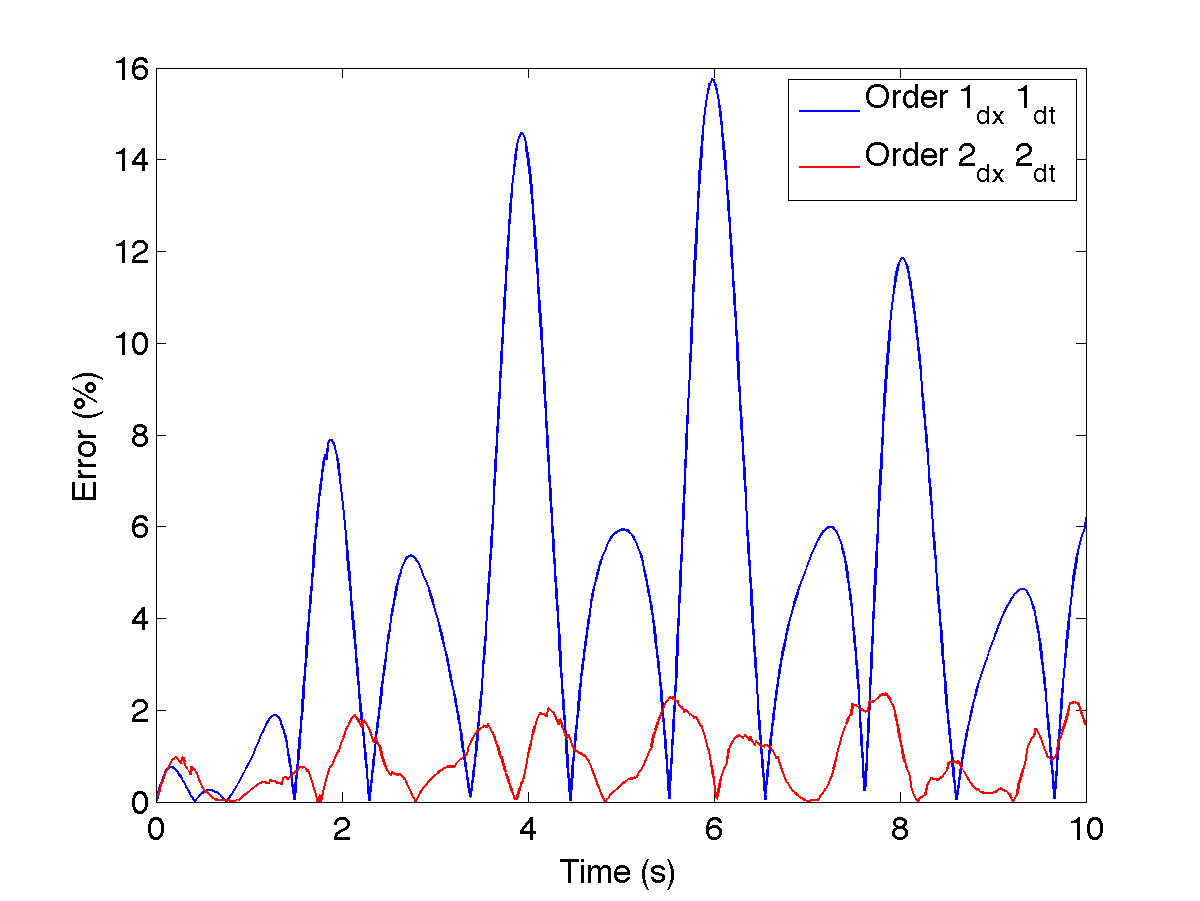}
\end{center}
\caption{Parabolic bowl: variations along time of the error at node $x_0=0.8$ m, 1$^{st}$ order
schemes (space and time) and 2$^{nd}$ order
schemes (space and time).}
\label{fig:conv_thacker_nh_x0}
\end{figure}

\subsection{The solitary wave}

As in the previous paragraph, we are able to examine the convergence
rate of the simulated solution towards the analytical one (given in
paragraph~\ref{subsec:solitary}).

We consider the analytical solution corresponding to the choices
$H_0=1$ m, $l = 1.7$ m, $d=1$ m. These choices lead to $a=0.5291$ m
and $c_0 = 3.873$ m.s$^{-1}$. We compare the analytical solution and
its simulated version at time $t=6$ s.

We have plotted in Fig.~\ref{fig:conv_soliton} the
$\mbox{log}(L^1-error)$ over the water depth at time $T=6$ seconds versus
$\mbox{log}(h_0/h_i)$ for the first and second-order scheme and they
are compared to the theoretical order. These errors have been computed on 7 meshes with $80$,
$120$, $200$, $400$, $800$, $1600$ and $3200$ cells.

For the curves obtained over Fig.~\ref{fig:conv_soliton}, the soliton
is~-- at the initial instant~-- in the fluid domain meaning the numerical treatment of
the boundary condition does not play a crucial role. Figure~\ref{fig:conv_soliton_ext} is similar to
Fig.~\ref{fig:conv_soliton} except that the soliton is not within
the fluid domain at the initial instant but enters the channel by the left
boundary. The convergence order of the scheme is examined at time
$T=10$ seconds and the soliton arrives within the fluid domain after 4 seconds
of simulation. Hence the soliton propagates during 6 seconds within
the domain corresponding to the same situation
as~Fig.~\ref{fig:conv_soliton}. Following~\ref{subsec:BC}, we have imposed a
given flux at the entry of the domain. We notice that the convergence orders obtained over Figs.~\ref{fig:conv_soliton}
and~\ref{fig:conv_soliton_ext} are similar.
\begin{figure}[hbtp]
\begin{center}
\includegraphics[width=10.5cm]{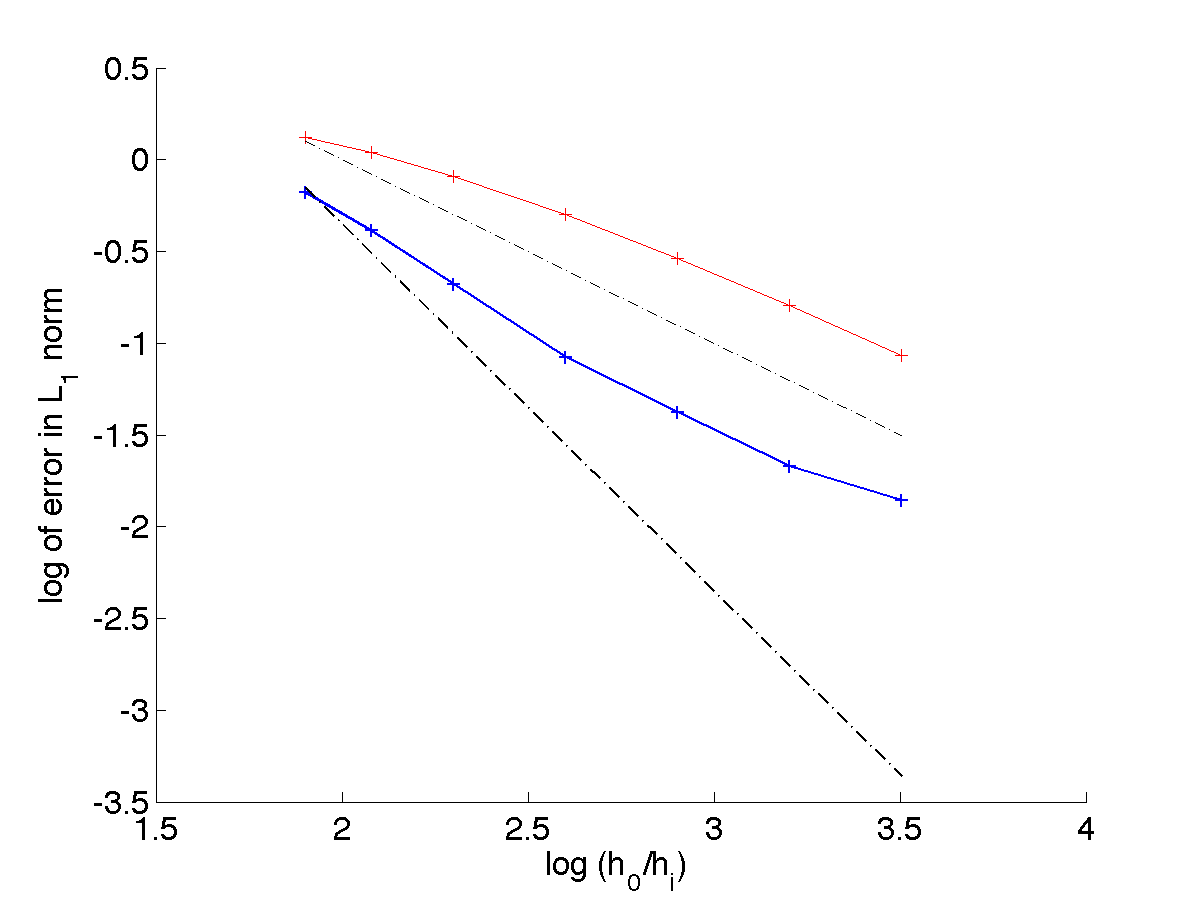}
\end{center}
\caption{Soliton (interior of the domain): convergence rates to the reference solution, 1$^{st}$ order
schemes (space and time) and 2$^{nd}$ order
schemes (space and time), '-.' theoretical order.}
\label{fig:conv_soliton}
\end{figure}
\begin{figure}[hbtp]
\begin{center}
\includegraphics[width=10.5cm]{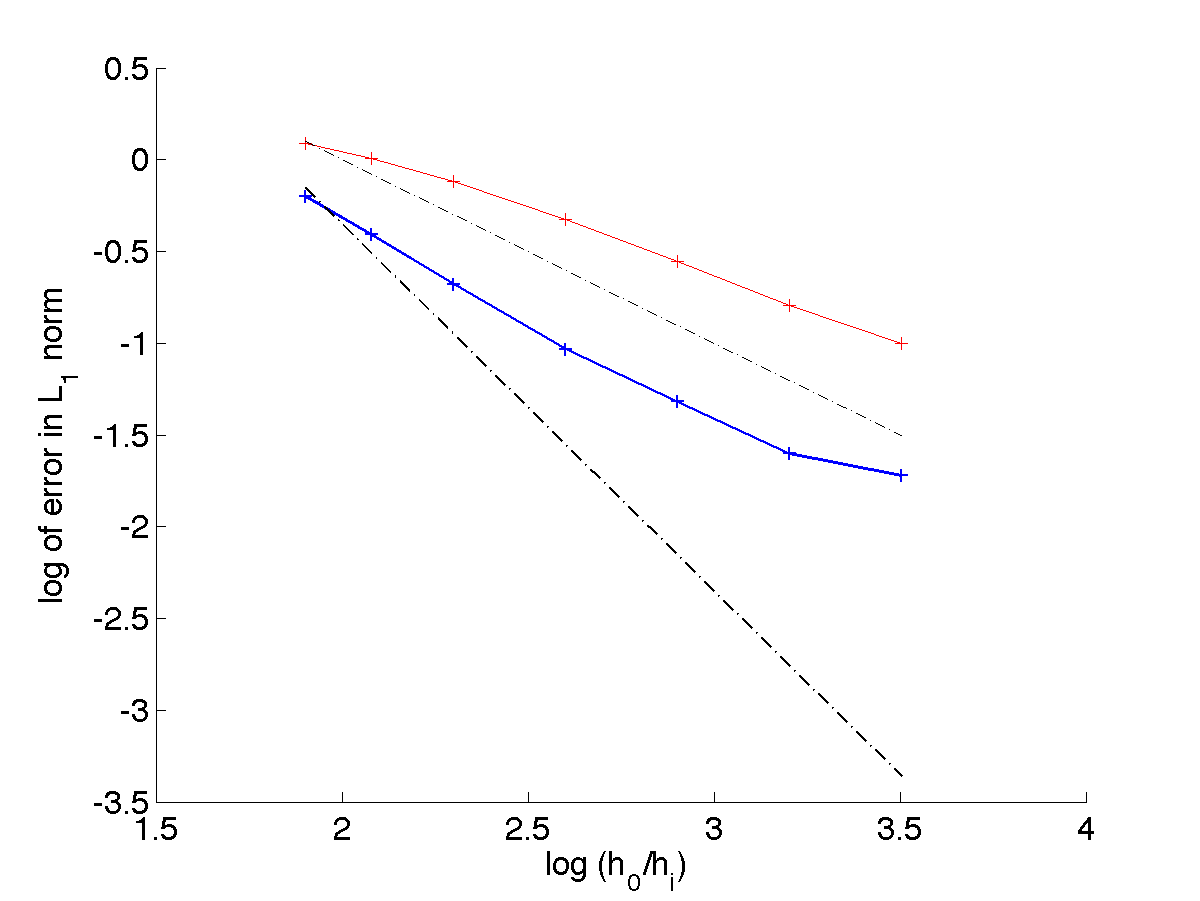}
\end{center}
\caption{Soliton (entering in the domain): convergence rates to the reference solution, 1$^{st}$ order
schemes (space and time) and 2$^{nd}$ order
schemes (space and time), '-.' theoretical order.}
\label{fig:conv_soliton_ext}
\end{figure}

\section{Conclusion}

In this paper we have proposed a robust and efficient numerical scheme
for a non-hydrostatic shallow water type model approximating the
incompressible Euler system with free surface.

The correction step is the key point of the scheme and
especially the discretization of the shallow water type version of the
gradient and divergence operators. A finite difference strategy has been used to discretize this
correction step. In order to be able to treat 2d flows on
unstructured meshes, a variational approximation of the correction step
is required, it is presented and its numerical performance is evaluated in~\cite{nora2}.

\section*{Acknowledgments}
The authors wish to express their warm thanks to Anne Mangeney for many fruitful discussions.

\bibliographystyle{amsplain}
\bibliography{boussinesq}

\end{document}

%% file: Figures/notations.pdf_t
\begin{picture}(0,0)%
\includegraphics{notations.pdf}%
\end{picture}%
\setlength{\unitlength}{4144sp}%
\begingroup\makeatletter\ifx\SetFigFont\undefined%
\gdef\SetFigFont#1#2#3#4#5{%
  \reset@font\fontsize{#1}{#2pt}%
  \fontfamily{#3}\fontseries{#4}\fontshape{#5}%
  \selectfont}%
\fi\endgroup%
\begin{picture}(9813,5565)(1021,-5923)
\put(7921,-3256){\makebox(0,0)[lb]{\smash{{\SetFigFont{20}{24.0}{\rmdefault}{\mddefault}{\updefault}{\color[rgb]{0,0,0}$u(x,z,t)\approx \overline{u}(x,t)$}%
}}}}
\put(10531,-1726){\makebox(0,0)[lb]{\smash{{\SetFigFont{20}{24.0}{\rmdefault}{\mddefault}{\updefault}{\color[rgb]{0,0,0}$x$}%
}}}}
\put(1036,-601){\makebox(0,0)[lb]{\smash{{\SetFigFont{20}{24.0}{\rmdefault}{\mddefault}{\updefault}{\color[rgb]{0,0,0}$z$}%
}}}}
\put(5716,-1276){\makebox(0,0)[lb]{\smash{{\SetFigFont{20}{24.0}{\rmdefault}{\mddefault}{\updefault}{\color[rgb]{0,0,0}Free surface}%
}}}}
\put(2836,-4741){\makebox(0,0)[lb]{\smash{{\SetFigFont{20}{24.0}{\rmdefault}{\mddefault}{\updefault}{\color[rgb]{0,0,0}$z_b(x,t)$}%
}}}}
\put(5356,-3391){\makebox(0,0)[lb]{\smash{{\SetFigFont{20}{24.0}{\rmdefault}{\mddefault}{\updefault}{\color[rgb]{0,0,0}$H(x,t)$}%
}}}}
\put(5986,-5371){\makebox(0,0)[lb]{\smash{{\SetFigFont{20}{24.0}{\rmdefault}{\mddefault}{\updefault}{\color[rgb]{0,0,0}Bottom}%
}}}}
\put(1036,-1996){\makebox(0,0)[lb]{\smash{{\SetFigFont{20}{24.0}{\rmdefault}{\mddefault}{\updefault}{\color[rgb]{0,0,0}$0$}%
}}}}
\put(3336,-1666){\makebox(0,0)[lb]{\smash{{\SetFigFont{20}{24.0}{\rmdefault}{\mddefault}{\updefault}{\color[rgb]{0,0,0}$H(x,t)+z_b(x)$}%
}}}}
\end{picture}%